\documentclass[11pt, reqno]{article}
\usepackage{amsmath}  
\usepackage{amssymb}
\usepackage{amsfonts}
\usepackage{amsthm}

\usepackage{mathtools}

\usepackage[thinc]{esdiff}

\usepackage{epstopdf}
\usepackage[caption=false]{subfig}
\theoremstyle{plain}

 \newtheorem{Theorem}{Theorem}[section]
 \newtheorem{Corollary}[Theorem]{Corollary}
 \newtheorem{Lemma}[Theorem]{Lemma}
 
 \theoremstyle{definition}
 \newtheorem{Definition}[Theorem]{Definition}
 \theoremstyle{remark}

 \numberwithin{equation}{section}

\newtheorem*{proofn}{Proof}

\usepackage{enumerate}

\usepackage{caption}

\usepackage[colorlinks=true, all colors=blue]{hyperref}

\usepackage{hyperref}

\usepackage{authblk}

\makeatletter
\renewcommand*{\@fnsymbol}[1]{\@alph{#1}}

\makeatother

\usepackage{physics}

\usepackage{derivative}

\usepackage[pdf]{pstricks}

\title{\bf{Differentiability of the largest Lyapunov exponent for non-planar open billiards}}

\author {Amal Al Dowais 
 \thanks {Department of Mathematics, College of Science and Arts, Najran University, Najran, Saudi Arabia \newline{\textit{\indent ~~ Email address:}\href{mailto: amalduas@nu.edu.sa} {amalduas@nu.edu.sa}} }
  }

\affil{\footnotesize{Department of Mathematics and Statistics, School of Physics, Mathematics and Computing, University of Western Australia, Perth, WA 6009, Australia. {\textit{ Email address:} \href{mailto: amal.aldowais@research.uwa.edu.au}{amal.aldowais@research.uwa.edu.au}}}}

\date{}

\begin{document}

\maketitle






\begin{abstract}
This paper investigates the behaviour of open billiard systems in high-dimensional spaces. Specifically, we estimate the largest Lyapunov exponent, which quantifies the rate of divergence between nearby trajectories in a dynamical system. This exponent is shown to be continuously dependent and differentiable with respect to a small perturbation parameter. A theoretical analysis forms the basis of the investigation. Our findings contribute to the field of dynamical systems theory and have significant implications for the stability of open billiard systems, which are used to model physical phenomena. The results provide a deeper comprehension of the behaviour of open billiard systems in high-dimensional spaces and emphasise the importance of taking small perturbations into consideration when analysing these systems.
\end{abstract}

{\small \bf {keywords.}} {\small {Open billiards; Lyapunov exponents; Billiard deformation}}


\section{Introduction}

In this paper, we continue our investigation of regularity properties of Lyapunov exponents from \cite{Amal} for open billiards in Euclidean spaces.
In \cite{Amal} we studied the case of open billiards in the plane; here we deal with the higher-dimensional case.
Open billiards are a type of dynamical system in which a particle travels at a constant velocity and collides with the boundaries of an unbounded domain.  Based on geometrical optics principles and satisfying Ikawa's no-eclipse condition ({\bf H}), the convex hull of a finite number of strictly convex compact obstacles does not intersect another obstacle. The study of open billiards has led to important insights into the behaviour of chaotic systems, as small changes in initial conditions can lead to vastly different outcomes. One way to quantify the behaviour of chaotic systems is through the use of Lyapunov exponents. Lyapunov exponents measure the rate at which neighbouring trajectories in a dynamical system diverge or converge, revealing the growth or decay of small perturbations. In the case of open billiards, the hyperbolic non-wandering set of the billiard map indicates the presence of positive and negative Lyapunov exponents.
Numerous studies have investigated Lyapunov exponents for billiards, (see \cite{Wo}, \cite{Ba},  \cite{Ch2}, \cite{Mar1}, \cite{Mar2}, \cite{ChM1}). This paper aims to estimate the largest Lyapunov exponent for higher-dimensional open billiards. The paper demonstrates that the largest Lyapunov exponent's dependency on the billiard deformation parameter $\alpha$ is continuous. Furthermore, the largest Lyapunov exponent is differentiable with respect to the deformation parameter $\alpha$. This work generalises previous results on the largest Lyapunov exponent for planar open billiards \cite{Amal}.

\bigskip

Our primary findings are as follows:

\bigskip

\noindent The largest Lyapunov exponent for open billiard in $\mathbb{R}^n$ is 

$$\lambda_1 = \lim_{m\to \infty}  \frac{1}{m} \sum_{i=1}^{m} \log {\Big(1+d_i (0) \ell_i (0)\Big)},$$

\noindent where $d_i$ is the distance between two reflection points along trajectories, and $\ell_i$ is related to the curvature of the unstable manifold, is defined in (\ref{L}). 

\bigskip

To comprehend the following theorems, an understanding of non-planar billiard deformations and their concepts is required. A small deformation parameter $\alpha \in [0, b], b \in \mathbb{R}$ characterises billiard deformations, which include position, rotation, and obstacle reshaping. The initial boundary parameterisation and deformation parameter $\alpha$ are used to parameterise the obstacle borders. This was introduced in \cite{W1}. Billiard deformations and notations in the subsequent theorems are explained in Section \ref{defor}.
The study assumes that billiard deformation is a differentiable function for both parameters, providing the foundation for the theorems.

\begin{Theorem} 
Let $K(\alpha)$ be a $\mathcal{C}^{4, 1}$ billiard deformation in $\mathbb{R}^n, n \geq 3$. Let $\lambda_1(\alpha)$ be the largest Lyapunov exponent for $K(\alpha)$. Then the largest Lyapunov exponent is continuous as a function of $\alpha$.
\label{cont2.}
\end{Theorem}

\begin{Theorem}
Let $K(\alpha)$ be a $\mathcal{C}^{5,2}$ billiard deformation in $\mathbb{R}^n, n \geq 3$. Let $\lambda_1(\alpha)$ be the largest Lyapunov exponent for $K(\alpha)$. Then $\lambda_1 (\alpha)$ is $\mathcal{C}^1$ with respect to $\alpha$.

\label{diff2.}

\end{Theorem}

\noindent The proofs of these theorems in Sections \ref{ContL} and \ref{DiffL} respectively.

\section{Preliminaries}

\subsection{Open billiards}

Let $K_{1}, K_2, \dots, K_{z_0}$ be strictly convex, disjoint, and compact subsets of $\mathbb{R}^n, n \geq 3$ with smooth boundaries $\partial K_i$, and satisfying condition {\textbf {(H)}} of Ikawa \cite{I}: for any $ i \neq j \neq k$ the convex hull of $ K_{i} \cup K_{k}$ does not have any common points with $K_{j}$. Let $\Omega $ be the exterior of $K$ (i.e., $ \Omega = \overline{\mathbb{R}^n \backslash  K}$), where $K= \bigcup_i K_i$. Let $$\mathcal{M} =  \{x =(q,v) \in int \, \, (\Omega) \times \mathbb{S}^{n-1} \, \, or \,\, (q,v) \in  \partial \Omega \times \mathbb{S}^{n-1} : \langle \nu_{\partial{K}}(q),v \rangle \geq 0 \},$$

\noindent where $n_{\partial{K}}(q)$ is the outwards unit normal vector to $\partial K$ at $q$. For $t \in \mathbb{R}$ and $x \in \mathcal{M}$ , the billiard flow $\Phi_t$ is defined as $\Phi_t(x) = (q_t, v_t)$, where $q_t$ and $v_t$ represent the position and velocity of the $x$ at time $t$. Let $\Lambda$ represent the set of all points of $\mathcal{M}$ that have bounded trajectories.
Let 
$$ M = \{x=(q,v) \in (\partial K \times \mathbb{S}^{n-1}) \cap \mathcal{M} : \langle \nu(q),v \rangle \geq 0 \}.$$
 
 \noindent Let $t_j \in\mathbb{R}$, where $j \in \mathbb{Z}$, represent the time of $j$-th reflection. The {\it billiard map} $B : M \to M$ is defined as $B(q_0, v_0) = (q_1, v_1)$ where $q_1 = q_0 +t_1 v_0 \in \partial K_j$ and $v_1= v_0 - 2 \langle v_0, \nu(q_1) \rangle \, \nu(q_1)$. Clearly, $B$ is a smooth diffeomorphism on $M$. Define the {\it canonical projection map} $\pi : M \to \partial {K}$ by $\pi(q,v) = q$. Let $M_0 \subset M$ be the {\it non-wandering set} of the billiard map $B$, that is, $M_0 = \{x \in M : |t_{j}(x)| < \infty , \,for \, all\, j \in \mathbb{Z}\}$. It is clear the $M_0$ is invariant under $B$. 
See \cite{Si1}, \cite{Si2}, \cite{Ch2}, \cite{ChM1}, \cite{PS}, for general information about billiard dynamical systems.

\subsection{Symbolic coding for open billiards}

\label{coding}

Each particular $x \in M_0$ can be coded by a {\it bi-infinite sequence} $$ \xi(x) =  (...,\xi_{-1},\xi_0,\xi_1,...) \in \{1,2,...,z_0\}^\mathbb{Z},$$ in which $ \xi_i \neq \xi_{i+1}$, for all $i \in \mathbb{Z}$, and $ \xi_j $ indicates the obstacle $K_{\xi_j}$ such that $\pi B^j(x) \in \partial K_{\xi_j}$. For example, if there are three obstacles $K_1,K_2, K_3$ and $K_4$ as above and a particular $q$ repeatedly hits $K_2,K_1, K_4, K_3,K_2, K_1, K_4, K_3$, then the bi-infinite sequence is $(...,2,1,4,3,2,1,4,3,...)$. Let $\Sigma $ be the {\it symbol space} which is defined as: $$\Sigma = \{ \xi = (...,\xi_{-1},\xi_0,\xi_1,...) \in \{1,2,...,z_0\}^\mathbb{Z} : \xi_i \neq \xi_{i+1} , \forall i \in \mathbb{Z} \}.$$ Define the {\it representation map} $ R: M_0 \to \Sigma$ by $R(x) = \xi(x)$. Let $\sigma: \Sigma \to  \Sigma $ be the {\it two-sided subshift} map 
defined by $\sigma(\xi_i) = \xi'_i$ where $\xi'_i = \xi_{i+1}$. 

\noindent 
It is known that the representation map $R: M_0 \to \Sigma$ is a homeomorphism (see e.g. \cite{PS}). See \cite{I}, \cite{LM}, \cite{Mor}, \cite{PS}, \cite{St1},  for topics related to symbolic dynamics for open billiards.

\subsection{Lyapunov exponents}


\noindent For the open billiard $B: M_0 \longrightarrow M_0$ in $\mathbb{R}^n$
we will use the coding $R: M_0 \longrightarrow \Sigma$ from Section \ref {coding}, which conjugates
$B$ with the shift map $\sigma : \Sigma \longrightarrow \Sigma$,
 to define Lyapunov exponents. It follows from the symbolic coding that there are ergodic $\sigma$-invariant measures $\mu$ on $\Sigma$. Let $\mu$ be an ergodic  $\sigma$-invariant probability measure on $\Sigma$.
The following is a consequence of Oseledets Multiplicative Ergodic Theorem (see e.g. \cite{V}, \cite{KS}):

\begin{Theorem} [A Consequence of Oseledets Multiplicative Ergodic Theorem]

There exist real numbers $\lambda_1 > \lambda_2 > \ldots > \lambda_k > 0 > -\lambda_k > \ldots > -\lambda_1$ 
and vector subspaces $E^u_j(x)$ and  $E^s_j(x)$ of $T_x(M)$, $1 \leq j \leq k \leq n$, $x \in M_0$,
depending measurably on $R(x) \in \Sigma$ such that:

\begin{enumerate} 

\item $E^u(x) = E^u_1(x)  \oplus \ldots \oplus E^u_k(x)$ and $E^s(x) = E^s_1(x)  \oplus \ldots \oplus E^s_k(x)$
for almost all $x \in M$;

\item $D_x B ( E^u_i(x)) = E^u_i(B(x))$ and $D_x B (E^s_i(x)) = E^s_i(B(x))$
for all $x\in M_0$ and all $i = 1, \ldots, k$, and

\item For almost all $x \in M_0$ there exists
$$\lim_{m\to\infty} \frac{1}{m} \log \|D_x B^m (w)\| = \lambda_j$$

\noindent whenever $w \in E^u_j(x) \oplus \ldots \oplus E^u_k(x)$ ($j \leq k$) however
$w \notin  E^u_{j+1}(x)  \oplus \ldots \oplus E^u_{k}(x)$.

\bigskip 

And $$\lim_{m\to\infty} \frac{1}{m} \log \|D_x B^m (w)\| = - \lambda_j$$

\noindent whenever $w \in E^s_j(x) \oplus \ldots \oplus E^s_k(x)$ ($j \leq k$) however
$w \notin  E^s_{j+1}(x)  \oplus \ldots \oplus E^s_{k}(x)$.

\end{enumerate} 
\label{OSE}
\end{Theorem}

\noindent Here "for almost all $x$" means "for almost all $x$" with respect to $\mu$. The numbers $\lambda_1, \ldots, \lambda_k$ are called {\it Lyapunov exponents}, while the invariant 
subspaces $E_j(x)$ are called {\it Oseledets subspaces}.

 
\subsection{Propagation of unstable manifolds for open billiards}


This part explains the relationship between unstable manifolds for the billiard ball map and the billiard flow in $\mathbb{R}^n$. Recall that $M_0$ represents the non-wandering set of the billiard ball map, while $\Lambda$ represents the non-wandering set of the billiard flow. For the billiard map, the unstable manifolds of size $\epsilon$ is 

$$W_\epsilon^u(x) = \{ y\in M : d (B^{-n}(x), B^{-n}(y)) \leq \epsilon \: \, \mbox{\rm for all }
\: n \in \mathbb{N} \; , \: d (B^{-n}(x),B^{-n}(y)) \to_{n\to \infty} 0\: \}.$$

\noindent Similarly, for $x = (q,v) \in \Lambda$ the unstable manifolds  $\widetilde{W}_\epsilon^u(x) $ for the billiard ball flow is 
\[\widetilde{W}_\epsilon^u(x) = \{ y\in \mathcal{M}  : d (\Phi_t(x),\Phi_t(y)) \leq \epsilon \: \, \mbox{\rm for all }
\: t \leq 0 \; , \: d (\Phi_t(x),\Phi_t(y)) \to_{t\to -\infty} 0\: \}. \]

\noindent It is known that the unstable manifolds for the billiard ball map and the billiard flow naturally are related. This correspondence can be described as follows geometrically. Given a point $x_0 = (q_0,v_0)$ in the non-wandering set $M_0$, and a small number $0<t^{(1)} < t_1$, let $y_0 = (q_0+t^{(1)}v_0,v_0) $. Then, there exists a one-to-one correspondence map $\tilde{\varphi}_1$between the unstable manifold $W^u(x_0)$ and the unstable manifold $\widetilde{W}^u(y_0)$. 
In addition, it follows from Sinai \cite{Si1}, \cite{Si2}, that the unstable manifold $\widetilde{W}^u(y_0)$ has the form $$\widetilde{W}^u(y_0) = \{ (p,n_{Y_0}(p)) : p\in Y_0\},$$

\noindent where $Y_0$ is a smooth $(n-1)$-dimensional hypersurface in $\mathbb{R}^n$ containing the point $y_0$ and is strictly convex with respect to the unit normal field $n_{Y_1}(p)$. 

Likewise, for all $j =1, 2, \ldots, m$, there exists  one-to-one correspondence $\tilde{\varphi}_j$ between the unstable manifolds $W^u(x_j)$ and $\widetilde{W}^u(y_j)$, where $B^j (x_0) = x_j = (q_j , v_j)$ and  $y_j = (q_j + t^{(j)} v_j , v_j)$, for a small positive $t_{j-1} < t^{(j)} < t_j - t _{j-1}$. Moreover, $\widetilde{W}^u(y_j)$ takes the form $ \tilde{Y_j}= \{(p_j, n_{Y_j}(p_j)) ; j \in {Y_j}\}$, where  $Y_j$ is also a smooth $(n-1)$-dimensional hypersurface in $\mathbb{R}^n$ containing the point $y_j$ and is strictly convex with respect to the unit normal field $n_{Y_j}(p_j)$.  



\bigskip

The following are the commutative diagrams involving the unstable manifolds for open billiard maps and for the billiard flows, as shown above:

$$\def\normalbaselines{\baselineskip20pt\lineskip3pt \lineskiplimit3pt}
\def\mapright#1{\smash{\mathop{\longrightarrow}\limits^{#1}}}
\def\mapdown#1{\Big\downarrow\rlap{$\vcenter{\hbox{$\scriptstyle#1$}}$}}
\begin{matrix}
W^u(x_0)  &\mapright{B(x_0)}& W^u(x_1) & \mapright{B^2(x_0)}& W^u(x_2)  & \ldots \ldots  & W^u(x_{m-1})  &\mapright{B^m(x_0)}& W^u(x_m) \cr 
\mapdown{\tilde{\varphi}_0}& & \mapdown{\tilde{\varphi}_1} & &\mapdown{\tilde{\varphi}_2} & & \mapdown{\tilde{\varphi}_{m-1}} & & \mapdown{\tilde{\varphi}_m} 
\cr \widetilde{W}^u(y_0) &\mapright{\Phi_{\tau_1}(y_0)} & \widetilde{W}^u(y_1) &\mapright{\Phi_{\tau_2}(y_0)} & \widetilde{W}^u(y_{2})  & \ldots \ldots & \widetilde{W}^u(y_{m-1})  &\mapright{\Phi_{\tau_m}(y_0)}& \widetilde{W}^u(y_m)
\end{matrix}
$$ 

\bigskip 

\noindent where ${\tau_j} = t_j + t^{(j+1)}$. Similarly, the following are the commutative diagrams involving the corresponding tangent spaces of unstable manifolds under the derivative of the billiard ball maps and the derivative of the billiard flows:

$$\def\normalbaselines{\baselineskip20pt\lineskip3pt \lineskiplimit3pt}
\def\mapright#1{\smash{\mathop{\longrightarrow}\limits^{#1}}}
\def\mapdown#1{\Big\downarrow\rlap{$\vcenter{\hbox{$\scriptstyle#1$}}$}}
\begin{matrix}
E^u(x_0)  &\mapright{D B}& E^u(x_1) & \mapright{D B^2}& E^u(x_2)  & \ldots \ldots  & E^u(x_{m-1})  &\mapright{D B^m}& E^u(x_m) \cr 
\mapdown{D\tilde{\varphi}_0}& & \mapdown{D\tilde{\varphi}_1} & &\mapdown{D\tilde{\varphi}_2} & & \mapdown{D\tilde{\varphi}_{m-1}} & & \mapdown{D \tilde{\varphi}_m} 
\cr \widetilde{E}^u(y_0) &\mapright{D \Phi_{\tau_1}} & \widetilde{E}^u(y_1) &\mapright{D \Phi_{\tau_2}} & \widetilde{E}^u(y_{2})  & \ldots \ldots & \widetilde{E}^u(y_{m-1})  &\mapright{D \Phi_{\tau_m} }& \widetilde{E}^u(y_m)
\end{matrix}
$$ 

\bigskip

 \noindent where $ E^u(x_j) = T_{x_j}(W^u(x_j))$ and $\widetilde{E}^u(y_j) = T_{y_j}(\widetilde{W}^u(y_j))$, and $\widetilde{E}^u(y_j) \perp v_j$.  Since the derivatives $D\tilde{\varphi}_{j-1}$ and $D\tilde{\varphi}_j$, for all $j$,  are uniformly bounded \cite{Si1}, then there exist global constants  $C > c >0$ such that

\begin{equation} 
 c \|D\Phi_{\tau_m} \| \leq \| DB^m \| \leq C \| D\Phi_{\tau_m} \|. 
 \label{DBB}
 \end{equation} 
 
 \noindent This will be used to calculate the largest Lyapunov exponent $\lambda_1$ for the open billiard map in Section \ref{0L}.

\bigskip



Based on the prior discussion, we will now apply the concept of unstable manifold propagation to write the main theorem, which involves the propagation of an appropriate convex curve on a convex hypersurface. This theorem was proved in reference \cite{St3} and will be utilised to calculate the largest Lyapunov exponent, $\lambda_1$.




 
 \bigskip

Let $x =(q, v) \in M_0$ and let $W^u_\epsilon (x_0)$ be the local unstable manifold for $x_0$ for sufficiently small $\epsilon > 0$. Then $W^u_\epsilon (x_0) = \{(x, n_X(x)) : x \in X \}$, where
$X$ is a convex curve on a smooth hypersurface $\widetilde{X} $ containing $q$ such that  $\widetilde{X} $ is strictly convex with respect to the unit normal field $n_{ \widetilde{X} }$. This follows from (cf. \cite{Si1},\cite{Si2}), see also \cite{St3}.
Let $\mathcal{K}_x : T_q \widetilde{X} \to T_q \widetilde{X}$ be the curvature operator (second fundamental form). Since $\widetilde{X} $ is strictly convex, then the curvature $\mathcal{K}_x$ is positive definite with respect to the unit normal field $n_{ \widetilde{X} }$. 
 
 \bigskip

Let $X$ be parametrized by $q(s), s \in [0, a]$, such that $q(0) = q_0 + r v_0$ for a small $ r>0$, and by the unit normal field $n_X(q(s))$. Let $q_j(s), j = 1, 2, ..., m$ be the $j$-th reflection points of the forward billiard trajectory $\gamma(s)$ generated by $x(s) = (q(s), n_X(q(s))$. We assume that $a > 0$ is sufficiently small so that the $j$-th reflection points $q_j(s)$ belong to the same boundary component $\partial{K}_{\xi_j}$ for every $s\in [0,a]$. Let $0 < t_1(x(s)) < ... < t_m(x(s))$ be the times of the reflections of the ray $\gamma(s)$ at $\partial K$.
 Let $\kappa_j(s)$ be the curvature of $\partial K$ at $q_j(s)$ and $\phi_j(s)$ be the collision angle between the unit normal $\nu_j(s)$ to $\partial K$ and the reflection ray of $\gamma(s)$ at $q_j(s)$. Also, let $d_j(s)$ be the distance between two reflection points i.e. $d_j(s) =  \| q_{j+1}(s) - q_j(s)\|$, $j=0,1,\dots,m$.

\bigskip
 
\noindent Given a large $m\geq 1$, let $t_{m}(x(s)) < t < t_{m+1}(x(s))$. 
Set $\hat{X} = \{ (q(s), n_X(s)) : s\in [0,a] \}$, and $\Phi_{t} (\hat{X}) = \hat{X}_{t}$. Let $\pi (\Phi_{t} (x(s))) = p(s)$. Then $p(s), s\in [0, a]$, is a parameterization of the $\mathcal{C}^3 $ curve $X_{t} = \pi (\Phi_{t} (\hat{X})$. 

Let $k_j(s)$ be the normal curvature of $X_{t_j(s)} = \lim_{t \searrow t_{j}(s)} X_t$ at $q_j(s)$ in the direction $\hat{w}_j(s)$ $(\| \hat{w}_j(s) \| =1)$ of $\lim_{t \searrow t_j(s)} (d/ds')w_t(s'))_{|s'=s}$, where $\hat{w}_j(s) = \dot{q}_t(s) / \| \dot{q}_t(s) \|$.  For $j\geq 0$ let $$ \mathcal{K}_j (s) : T_{q_j(s)} (X_{t_j(s)}) \to T_{q_j(s)} (X_{t_j(s)}) $$ be the curvature operator of $X_{t_j(s)}$ at $q_j(s)$, and define $l_j(s) >0$ by 

\begin{equation}
\big (1+d_j(s) \ell_j(s) \big)^2 = 1 + 2d_j(s) k_j(s) + (d_j(s))^2 \| \mathcal{K}_j(s) \hat{w}_j(s) \|^2.
\label{L}
\end{equation}

\bigskip

Set
\begin{equation}
\delta_j(s) = \frac{1}{ 1 + d_j(s) \ell_j(s)} \:\:\:\: , \:\:\: 1\leq j \leq m\;.
\end{equation}

\noindent
\begin{Theorem} {\em{$\cite{St3}$}}
For all $s \in [0,a]$  we have

\begin{equation}
\| \dot{q}(s)\| = \| \dot{p}(s)\| \delta_{1}(s)\delta_{2}(s) \ldots \delta_m(s)\;.
\label{for.} 
\end{equation}

\label{ST theorem}

\end{Theorem}

\label{propg}

\subsection{Curvature of unstable manifolds}

\label{unstable}

Following \cite{Si1}, \cite{Si2}, \cite{Ch1} ( see also \cite{BCST}), here we express  the curvature operator $\mathcal{K}_j$ of the convex front  $X_j$ at $q_j(s)$ using certain related objects:

\bigskip

• $\widetilde{X}^{-}_j$ represents the convex front passing $q_j(s)$ before collision, that is $\widetilde{X}^{-}_j = \lim_{\tau \nearrow t_j} \widetilde{X}_\tau$, where $ t_{j-1} < \tau < t_j$.  

• $\widetilde{X}^{+}_j$ represents the convex front passing $q_j(s)$ after collision, that is $\widetilde{X}^{+}_j = \lim_{\tilde{\tau} \searrow t_j} \widetilde{X}_{\tilde{\tau}}$, where $ t_{j} < \tilde{\tau} < t_{j+1}$. We write $\widetilde{X}_j$ to indicate to $\widetilde{X}^{+}_j$.

• $\mathcal{J}^{-}$ is the hyperplane of $\widetilde{X}^{-}_j$ at $q_j(s)$, i.e. $\mathcal{J}^{-} = T_{q_j} (\widetilde{X}^{-}_j)$, which is perpendicular to $v_{j-1}$.

• $\mathcal{J}^{+}$ is the hyperplane of $\widetilde{X}^{+}_j$ at $q_j(s)$, i.e. $\mathcal{J}^{+} = T_{q_j} (\widetilde{X}^{+}_j)$, which is perpendicular to $v_{j}$.

• $\mathcal{T}$ is the hyperplane of $\partial{K}$ at $q_j(s)$, i.e. $\mathcal{T} = T_{q_j} (\partial{K})$, which is perpendicular to $\nu_{j}(q_j(s))$.

\bigskip 

\noindent Now the curvature operator $\mathcal{K}_j : \mathcal{J}^{+} \to \mathcal{J}^{+}$ is given by

\begin{equation}
\mathcal{K}_{j} = U^{-1}_j \mathcal{K}^{-}_{j}  U_j+ 2 \cos\phi_{j} V^*_{j} N_{j} V_{j}, 
\label{B1}
\end{equation}

\noindent where 

• $\mathcal{K}^{-}_{j}$ is the curvature operator of $\widetilde{X}^{-}_j$, which defined as $\mathcal{K}^{-}_{j} = \mathcal{K}_{j-1}(I + d_{j-1}\mathcal{K}_{j-1})^{-1}$ or we can write $\mathcal{K}^{-}_{j} = (\mathcal{K}^{-1}_{j-1} + d_{j-1}I)^{-1}$.

• The unitary operator $U_j : \mathcal{J}^{+}_j \to \mathcal{J}^{-}_j$ is a projection parallel to $\nu_j(q_j(s))$, defined as; for all $w_j \in \mathcal{J}^{+}_j$ 

\begin{equation}
U_j w_j = w_j - \frac{\langle w_j, v_{j-1} \rangle } {\cos \phi_j } \nu_j . 
\label{U}
\end{equation}

• $U^{-1}_j : \mathcal{J}^{-}_j \to \mathcal{J}^{+}_j$ is a projection parallel to $\nu_j(q_j(s))$, defined as; for all $w_{j-1} \in \mathcal{J}^{-}_j$ 

\begin{equation}
U^{-1}_j w_{j-1} = w_{j-1} - \frac{\langle w_{j-1}, v_{j} \rangle } {\cos \phi_j } \nu_j . 
\label{U-1}
\end{equation}

• $V_j : \mathcal{J}^{+}_j \to \mathcal{T}_j$ is a projection parallel to $v_j$, defined as; for all $w_{j} \in \mathcal{J}^{+}_j$ 

\begin{equation}
V_j w_{j} = w_{j} - \frac{\langle w_{j}, \nu_{j} \rangle } {\cos \phi_j } v_j . 
\label{V}
\end{equation}

• $V^{*}_j : \mathcal{T}_j \to \mathcal{J}^{+}_j$ is a projection parallel to $\nu_j(q(s))$, defined as; for all $\tilde{u}_{j} \in \mathcal{T}_j$ 

\begin{equation}
V^*_j \tilde{u}_{j} = \tilde{u}_{j} - \frac{\langle \tilde{u}_{j}, v_{j} \rangle } {\cos \phi_j } \nu_j . 
\label{V*}
\end{equation}

• ${N}_{j}$ is the curvature operator s.f.f. of $\partial{K}$ at $q_j(s)$.

\bigskip

The operator $\| V^*NV \|$ is bounded by  $ \kappa_{\min} \leq  \| V^*NV \| \leq \frac { \kappa_{\max}}{\cos^2 \phi _j}$ where $\kappa_{\min}$ and $\kappa_{\max}$ are the minimum and maximum eigenvalues of the normal curvature $N$ and again $\phi_j(s)$ is the collision angle at $q_j(s)$ and $\phi_j(s) \in [0, \pi/2]$. Let $\mu_j(s) = \mu_j(x(s))$ and $\eta_j(s) = \eta_j(x(s))$ be the eigenvalues of the curvature operator $\mathcal{K}_j$. Then by using these and (\ref{B1}), we get 

\begin{equation}
\mu_{\min} \leq 2 \kappa_{\min} \leq \mu_{j+1}(s) \leq \eta_{j+1}(s) \leq \frac{ \eta_j(s)} {1+d_j(s) \eta_j(s) } + 2 \frac{ \kappa_{\max}} {\cos \phi_j(s) } \leq \frac{1}{d_{\min}} + 2 \frac { \kappa_{\max}} { \cos \phi_{\max} } \leq \eta_{\max}.
\label{eta}
\end{equation}




\section {Estimation of the largest Lyapunov exponent for non-planar open billiards}

\label{0L}

Here we want to estimate the largest Lyapunov exponent $\lambda_1$ for non-planar open billiards . 
We will use 
Oseledets multiplicative ergodic theorem \ref{OSE} and theorem \ref{ST theorem}.

Assume that $\mu$ is an ergodic $\sigma$-invariant measure on $\Sigma$,
and let $x_0 = (q_0,v_0) \in M_0$ correspond to a typical point in $\Sigma$  with respect to $\mu$ via the
representation map $R$. That is as in Theorem \ref{OSE}, we have
$$\lambda_1 = \lim_{m\to\infty} \frac{1}{m} \log \|D_{x_0}B^m(w)\| ,$$ with $ w \in E^u_1(x) \oplus E^u_2(x) \oplus \dots E^u_{k}(x) \backslash E^u_2(x) \oplus E^u_3(x) \oplus \dots E^u_{k}(x)$.

\bigskip

Let $w = \dot{q}(s)$ as in Section \ref{propg}. And then by using (\ref{DBB}), there exist some global constants $c_1 > c_2 > 0$, independent of $x_0$, $X$, $m$, etc. such that
$$c_2 \|\dot{p}(s)\| \leq \|D_{x_0}B^m\| \leq c_1 \|\dot{p}(s)\|$$
for all $s\in [0,a]$. So, by (\ref{for.}),
$$\frac{c_2 }{\delta_1(0) \delta_2(0) \ldots  \delta_{m}(0)} \leq \|D_{x_0}B^m\| 
\leq \frac{c_1}{\delta_1(0) \delta_2(0) \ldots \delta_{m}(0)}$$

for all $s \in [0,a]$. Using this for $s = 0$, taking logarithms and limits as $m \to \infty$, we obtain\\
$$- \lim_{m\to\infty} \frac{1}{m} \log \left( \delta_1(0) \delta_2(0) \ldots  \delta_{m}(0)\right) 
\leq \lim_{m\to\infty}\frac{1}{m} \|D_{x_0}B^m\| \leq
- \lim_{m\to\infty} \frac{1}{m} \log \left( \delta_1(0) \delta_2(0) \ldots  \delta_{m}(0)\right) .$$
Hence,


\begin{equation}
\begin{aligned}
\lambda_1 &=\lim_{m\to \infty} - \frac{1}{m} \sum_{i=1}^{m} \log \delta_{i}(0) =\lim_{m\to \infty}  \frac{1}{m} \sum_{i=1}^{m} \log {\Big(1+d_i (0) \ell_i (0)\Big)}.
\label{g}
\end{aligned}
\end{equation}

\bigskip

 We can use (\ref{L}) and the (\ref{eta}) to estimate the upper and lower of the largest Lyapunov exponent $\lambda_1$ as follows:

 \begin{equation}
\big (1+d_j(s) \ell_j(s) \big)^2 \leq  1 + 2 d_j(s) \eta_j(s) + (d_j(s))^2 (\eta_j(s))^2 = \big (1+d_j(s) \eta_j(s) \big)^2.
\label{ell}
\end{equation}

 \noindent This implies that 

\begin{equation}
1+d_j(s) \ell_j(s)  \leq  1+d_j(s) \eta_j(s) \leq 1 + d_{\max} \eta_{\max}.
\label{ell1}
\end{equation}

\noindent In the same way, we get

\begin{equation*}
1+d_j(s) \ell_j(s)  \geq  1+d_j(s) \mu_j(s) \geq 1 + d_{\min} \mu_{\min}.
\label{ell2}
\end{equation*}

\noindent Therefore, 
$$  \log(1 + d_{\min} \mu_{\min})   \leq  \lambda_1 \leq \log(1 + d_{\max} \eta_{\max}) .$$

\section{Billiard deformations in $\mathbb{R}^n$}
\label{defor}

Let $\alpha \in I = [0, b]$, for some $b \in \mathbb{R}^+$, be a deformation parameter and let $\partial K_j(\alpha)$ be parametrized counterclockwise by $\varphi_j(u^{(1)}_j, u_j^{(2)},..., u_j^{(n-1)}, \alpha)$. Let \newline $q_j = \varphi_j(u^{(1)}_j, u_j^{(2)},..., u_j^{(n-1)}, \alpha) $ be a point that lies on $\partial K_j(\alpha)$. Denote the perimeter of $\partial K_i(\alpha)$ by $L_j(\alpha)$, and let $R_j=\{ (u_j^{(1)}, u_j^{(2)}, ..., u_j^{(n-1)}: \alpha \in I, u^{(t)}_j \in [0,L_j(\alpha)]\}$.

\bigskip

\begin{Definition} {\em $\cite{W1}$} For any $\alpha \in I = [0, b]$, let $K(\alpha)$ be a subset of $\mathbb{R}^n, n \geq 3$. For integers $r \geq 2, r' \geq 1$, we call $K(\alpha) $ a $\mathcal {C}^{r,r'}$-{{billiard deformation}} if the following conditions hold for all $\alpha \in I$:

\begin{enumerate}

\item $K(\alpha) = \bigcup_{i=1}^{z_0} K_i(\alpha)$ satisfies the no-eclipse condition $({\bf{H}})$.  

\item Each $\partial K_i(\alpha)$ is a compact, strictly convex set with $\mathcal{C}^r$ boundary, and $K_i(\alpha) \bigcap K_j(\alpha) = \phi$ for $i \neq j$.


\item For each $i=1, 2,..., m$ and all $p\in \partial K_i(\alpha)$, there is a rectangle $R_p \subset \mathbb{R}^{n-1}$ and a $\mathcal{C}^{r,r'}$ function $\varphi_j : R_p \to \mathbb{R}^n$, which is an orthonormal parametrisation of $\partial K_i(\alpha)$ at $p$. 

\item For all integers $0 \leq l \leq r, 0\leq l' \leq r'$ (apart from $l=l'=0$), there exist constants $C_\varphi ^{(l,l')} $ depending only on the choice of the billiard deformation and the parametrisation $\varphi_j$, such that for all integers $j=1,2,3,...,z_0$, 
$$  \Big \| \frac { \partial ^{l'}}{ \partial \alpha^{l' }} \nabla _i^l \varphi_j \Big \| \leq {C}^{(l,l')}_ \varphi. $$

\end{enumerate}

\label{44}
\end{Definition}

We consider the open billiard deformation map, denoted as $B_\alpha$, defined on the non-wandering set $M_\alpha$ for $K(\alpha)$. As in Section \ref{coding}, we define $\Sigma_\alpha$ and $R_\alpha$ as the mapping from $M_\alpha$ to $\Sigma_\alpha$ such that $R_\alpha(x(\alpha)) = \xi (x(\alpha))$. Using the parameterization defined earlier, we can express the point corresponding to deformed billiard trajectories as $q_{\xi_j} (\alpha) = \varphi_{\xi_j} (u_j^{(1)}(\alpha), u_j^{(2)}(\alpha), ..., u_j^{(n-1)}(\alpha), \alpha) \in \partial{K}_{\xi_j}   (\alpha)$, where $u^{(t)}_{\xi_j}  (\alpha) \in [0, L_{\xi_j}  (\alpha)]$. We will write $q_j (\alpha) = \varphi_j (u_j^{(1)}(\alpha), u_j^{(2)}(\alpha), ..., u_j^{(n-1)}(\alpha), \alpha)$ for brevity.

\bigskip

\noindent It was shown in \cite{W1} that $u^{(t)}_j(\alpha) = u^{(t)}_{\xi_j}(\alpha)$, where $ t=1,2,...,n-1$, for a fixed $\xi \in \Sigma_\alpha$, is differentiable with respect to $\alpha$ and its higher derivative is bounded by a constant independent of $\alpha$ and $j$.

\begin{Theorem}{\em{\cite{W1}}}
Let $K(\alpha)$ be a $\mathcal{C}^{r,r'}$ billiard deformation with $r \geq 2, r' \geq 1$. Then for all $t = 1, 2, ..., n-1$, $u^{(t)}_j(\alpha)$ is $\mathcal{C}^{\min\{r-1,r'\}}$ with respect to $\alpha$, and there exist constants $C_u^{(s')} >0$ such that $$\Big \| \frac{\partial^{s'}{u^{(t)}_j(\alpha)}}{\partial \alpha^{s'}} \Big \|  \leq C^{(s')}_u.$$

\label{Du}
\end{Theorem}


\section{Propagation of unstable manifolds for non-planar billiard deformations}

\label{prop 1}

Here we want to re-describe the propagation of the unstable manifold mentioned in Section \ref{propg} for the billiard deformation $K(\alpha)$ with respect to the deformation parameter $\alpha \in I = [0, b]$. Let $K(\alpha)$ be a $\mathcal{C}^{r,r'}$ billiard deformation in $\mathbb{R}^n$, with $r\geq 4, r' \geq 1$. Let $\tilde{X}$ be a smooth hypersurface passing through $x_0 (\alpha)= (q_0(\alpha), v_0(\alpha)) \in M_\alpha$ and let $X$ be a $\mathcal{C}^3$ convex curve, with respect to the unit normal field $\nu_X$, on $\tilde{X}$. Let $X$ be parameterised by $q(s, \alpha)$ for all $s \in [0,a]$ and set $\hat{X}_\alpha = \{(q(s, \alpha), \nu_X(q(s,\alpha))), \forall s \in[0,a] \}$. For $j=1, 2, ..., m$, let $q_j(s, \alpha)$ be the reflection points generated by $x(s, \alpha) = (q(s, \alpha), \nu(q(s, \alpha)))$, $d_j(s, \alpha)$ be the distance between $q_{j}(s, \alpha)$ and $q_{j+1}(s, \alpha)$, and $\phi_j(s, \alpha)$ be the angle of reflection at $q_j(s, \alpha)$. 

\bigskip

Let $t_{m-1} (x(s, \alpha)) < t < t_{m}(x(s, \alpha))$ for some large $m \geq 2$. Let $\pi(\Phi_t(\hat{X}_\alpha)) = {X}_{\alpha_t}.$ Then, $X_{\alpha_t}$ is a $\mathcal{C}^3$ curve on $\tilde{X}_t$ parameterised by $u_t(s, \alpha)$. 

\bigskip 
 
Let $\mathcal{K}_j(s, \alpha)$ be the curvature operator of $X_{\alpha {t_j(s)}} = \lim_{t \searrow t_{j}(x(s))} X_{\alpha_t}$ at $q_j(s, \alpha)$ in the direction $\hat{w}_j(s, \alpha)$ of $\lim_{t \searrow t_j((x(s))} (d/ds')\hat{w}_t(s', \alpha'))_{|(s, \alpha)}$. Let $\mathcal{K}_0(s)$ be the curvature operator of $X$ at $q(s, \alpha)$, which is independent of $\alpha$. And, for $j\geq 1$ let 
$$ \mathcal{K}_j (s, \alpha) : T_{q_j(s, \alpha)} (X_{\alpha _{t_j(x(s))}}) \to T_{q_j(s, \alpha)} (X_{\alpha_{t_j(x(s,\alpha))}}) $$ 

\noindent be the curvature operator of $X_{t_j(x(s, \alpha))}$ at $q_j(s, \alpha)$. Define $\ell_j(s, \alpha)$ by

\begin{equation}
\big (1+d_j(s, \alpha) \ell_j(s, \alpha) \big)^2 = \| w_j(s, \alpha) + d_j(s, \alpha) \mathcal{K}_j(s, \alpha) \|^2 , \, \, \, \, \, \, {\textrm {for all}}\, \,  j= 0, 1,...,m .
\label{L 2}
\end{equation}

\noindent Then 

\begin{equation}
\big (1+d_j(s, \alpha) \ell_j(s, \alpha) \big)^2  = 1 + 2d_j(s, \alpha) k_j(s, \alpha) + (d_j(s, \alpha))^2 \| \mathcal{K}_j(s) \hat{w}_j(s) \|^2,
\label{L 3}
\end{equation}

\noindent where $k_j(s, \alpha)$ is the normal curvature of $X_{\alpha_j}(s, \alpha)$ at $q_j(s, \alpha)$ in the direction $\hat{w}_j(s, \alpha)$, which is given by $ k_j(s, \alpha) = k_j(\hat{w}_j(s, \alpha)) = \langle \mathcal{K}_j(s, \alpha) \hat{w}_j(s, \alpha),   \hat {w}_j(s, \alpha)\rangle$. 

Now, we want to re-write the curvature operator in (\ref{B1}) with respect to the billiard deformation parameter $\alpha$. So, we have for $j = 1, 2, ..., m$

\begin{equation}
\mathcal{K}_{j} (s, \alpha)= U^{-1}_j(s, \alpha)(\mathcal{K}^{-1}_{j-1} (s, \alpha)+ d_{j-1}(s, \alpha)I)^{-1} U_j(s, \alpha)+ 2 \cos\phi_j (s, \alpha)V^*_{j} (s, \alpha) N_{j} (s, \alpha) V_{j} (s, \alpha)
\label{K2}
\end{equation}

\bigskip

\noindent For brevity, we will write all previous characteristics as e.g. $\mathcal{K}(0, \alpha) = \mathcal{K}(\alpha)$, $d(0, \alpha)= d(\alpha)$, ..., in the case $s=0$.   We can write (\ref{K2}) as follows

\begin{equation}
\mathcal{K}_{j} (\alpha)= U^{-1}_j(\alpha)(\mathcal{K}^{-1}_{j-1} (\alpha)+ d_{j-1}(\alpha)I)^{-1} U_j(\alpha)+ 2 \Theta_j(\alpha),
\label{K3}
\end{equation}


\noindent where $ \Theta_j(\alpha) = {\cos\phi_j} {V}_{j}^* N_{j} {V}_{j}$.
 Note that all terms in the last formula are functions of  $\alpha$, and are defined as in Section \ref{unstable} with respect to $\alpha$.

\subsection{Estimates of the higher derivative of billiard characteristics in $\mathbb{R}^n$}

\label{derivative}


This section aims to demonstrate the differentiability of the billiard deformation characteristics in high dimensions with respect to $\alpha$. These characteristics are described in Section \ref{prop 1}. Furthermore, we establish that the derivatives of these characteristics are bounded by constants independent of deformation parameter $\alpha \in[0, b]$ and the number of reflections $j \in \mathbb{Z}^+$. In particular, we show that the first and second derivatives are bounded, which holds significant relevance for the subsequent Sections \ref{DiffL} and \ref{ContL}. The higher derivatives are bounded via induction. All corollaries, which are provided here, are based on {{Definition \ref{44}}} and Theorem {\ref {Du}}.


\bigskip





\begin{Corollary}
Let $K(\alpha)$ be a $\mathcal{C}^{r,r'}$ billiard deformation with $r \geq 2, r' \geq 1$, in $\mathbb{R}^n$. Let $q_j(\alpha)$ belong to $\partial K_{\xi_j}$. Then $q_j(\alpha)$ is $\mathcal{C}^{s'}$, where $s'=\min\{r-1,r'\}$, with respect to $\alpha$, and there exist constants $C_q^{(s')} >0$ such that $$\Big \| \odv[order=s']{q_j(\alpha)}{\alpha} \Big \|  \leq C^{(s')}_q.$$

\label{Dq}
\end{Corollary}

\begin{proof}
Let  $q_j (\alpha) = \varphi_j (u_j^{(1)}(\alpha), u_j^{(2)}(\alpha), ..., u_j^{(n-1)}(\alpha), \alpha) \in \partial K_{\xi_j}$. Then from {{Definition \ref{44}}} and Theorem {\ref {Du}}, $q_j (\alpha)$ is $\mathcal{C}^{\min\{r-1,r'\}}$ with respect to $\alpha$. For the first derivatives of $q_j(\alpha)$ we have 

\begin{equation*}
\begin{aligned}
\odv{q_j}{\alpha} &= \sum_{i=1}^{n-1} \pdv{ \varphi_{j}}{u_{j}^{(i)}} \pdv{ u_{j}^{(i)}}{ \alpha} + \pdv {\varphi_{j}}{ \alpha}.\\
\end{aligned} 
\end{equation*}

 \noindent From {{Definition \ref{44}}} and Theorem {\ref {Du}}, there exists $C_q^{(1)} > 0$ such that 

\begin{equation*}
\begin{aligned}\Big \| \odv{q_j}{\alpha} \Big \| & \leq (n-1) C_u^{(1)} + C_\varphi^{(0,1)} = C_q^{(1)}. \\
\end{aligned} 
\end{equation*}

\noindent For the second derivatives we have 

\begin{equation*}
\begin{aligned}
\odv[order=2]{q_j}{\alpha} &=  \sum_{i=1}^{n-1} \sum_{k=1}^{n-1} \pdv{ \varphi_{j}}{u_{j}^{(k)},  u_j^{(i)}} \pdv{ u_{j}^{(i)}} { \alpha} \pdv{ u_{j}^{(k)}}{ \alpha} + \sum_{i=1}^{n-1} \pdv{ \varphi_{j}}{u_{j}^{(i)}} \pdv[order=2]{ u_{j}^{(i)}} { \alpha}  + 2\sum_{i=1}^{n-1} \pdv{ \varphi_{j}}{u_{j}^{(i)}, \alpha} \pdv{ u_{j}^{(i)}}{ \alpha}+ \pdv[order=2]{\varphi_{j}}{ \alpha}. \\
\end{aligned} 
\end{equation*}

\noindent As before, there exists $C_q^{(2)} > 0$ such that 

\begin{equation*}
\begin{aligned}
\Big \| \odv[order=2]{q_j}{\alpha} \Big \| &\leq (n-1)^2 C_\varphi^{(2,0)} (C_u^{(1)})^2 + (n-1) C_u^{(2)} + 2 (n-1) C_\varphi^{(1,1)} C_u^{(1)} +C_\varphi^{(2,0)} = C_q^{(2)}. 
\end{aligned} 
\end{equation*}

\noindent This constant is independent of $j$ and $\alpha$. Continuing by induction, we can see that the $s'$-th derivative of $q_j(\alpha)$ is bounded by a constant that depends only on $s'$ and $n$. Thus, the statement is proved.  
\end{proof}


\begin{Corollary}
Let $K(\alpha)$ be a $\mathcal{C}^{r,r'}$ billiard deformation in $\mathbb{R}^n$ with $r \geq 2, r' \geq 1$. 
Then $d_j(\alpha)$ is $\mathcal{C}^{s}$, where $s=\min\{r-1,r'\}$, with respect to $\alpha$, and there exist constants $C_v^{(s)} >0$  depending only on $s,$ and $n$ such that $$\Big | \odv[order=s']{d_j(\alpha)}{\alpha} \Big |  \leq C^{(s')}_d.$$

\label{Dd}
\end{Corollary}

\begin{proof}

Let $d_j= \|q_{j+1}(\alpha) - q_{j}(\alpha)\|$, where $q_j(\alpha) = \varphi_j(u_j^{(1)}(\alpha), u_j^{(2)}(\alpha), ..., u_j^{(n-1)}(\alpha), \alpha) \in \partial K_{\xi_j}$. Then from Corollary \ref{Dq}, $d_j(\alpha)$ is $\mathcal{C}^{\min\{r-1,r'\}}$ and its derivative with respect to $\alpha$ is

\begin{equation*}
\begin{aligned}
\odv{d_j}{\alpha} & =  \, \Big < \frac {\varphi_{j+1} - \varphi_j } {\| \varphi_{j+1} - \varphi_j \|}  ,  \frac{\partial \varphi_{j+1}}{\partial \alpha} +  \sum_{i=1}^{n-1} \frac{\partial \varphi_{j+1}}{\partial u_{j+1}^{(i)}} \frac{ \partial u_{j+1}^{(i)}} { \partial \alpha} - \frac{\partial \varphi_{j}}{\partial \alpha} -  \sum_{i=1}^{n-1} \frac{\partial \varphi_{j}}{\partial u_{j}^{(i)}} \frac{ \partial u_{j}^{(i)}} { \partial \alpha} \Big >.\\
\end{aligned}
\end{equation*}

\noindent  By using the estimations in {{Definition \ref{44}}} and Theorem {\ref {Du}}, we get

\begin{equation*}
\begin{aligned}
 \Big | \odv{d_j}{\alpha} \Big | & \leq 2 C_\varphi ^{(0,1)} + 2 (n-1) C_u^{(1)} = C_d^{(1)},
\end{aligned}
\end{equation*}

\noindent where $ C_d^{(1)} >0$ is a constant depending only on $n$. And its second derivative is 

\begin{equation*}
\begin{aligned}
%
%
%
%
\odv[order=2]{d_j}{\alpha} &= \frac{\Big ( \frac{\partial \varphi_{j+1}}{\partial \alpha} +  \sum_{i=1}^{n-1} \frac{\partial \varphi_{j+1}}{\partial u_{j+1}^{(i)}} \frac{ \partial u_{j+1}^{(i)}} { \partial \alpha} - \frac{\partial \varphi_{j}}{\partial \alpha} -  \sum_{i=1}^{n-1} \frac{\partial \varphi_{j}}{\partial u_{j}^{(i)}} \frac{ \partial u_{j}^{(i)}} { \partial \alpha}  \Big )^2 } { \| \varphi_{j+1} - \varphi_j \|} \\ %
& \, \, \, - \frac{ (  \varphi_{j+1} - \varphi_j   )^2 \Big ( \frac{\partial \varphi_{j+1}}{\partial \alpha} +  \sum_{i=1}^{n-1} \frac{\partial \varphi_{j+1}}{\partial u_{j+1}^{(i)}} \frac{ \partial u_{j+1}^{(i)}} { \partial \alpha} - \frac{\partial \varphi_{j}}{\partial \alpha} -  \sum_{i=1}^{n-1} \frac{\partial \varphi_{j}}{\partial u_{j}^{(i)}} \frac{ \partial u_{j}^{(i)}} { \partial \alpha}  \Big )^2 } { \| \varphi_{j+1} - \varphi_j \|^3} \\
&  \, \, \, +  \frac {\varphi_{j+1} - \varphi_j } { \| \varphi_{j+1} - \varphi_j \| } \cdot \Biggl (\frac{ \partial^2 \varphi_{j+1}} {\partial \alpha^2 } + \sum_{k=1}^{n-1} \sum_{i=1}^{n-1}  \frac{\partial^2 \varphi_{j+1} } {\partial u_{j+1}^{(k)} \partial u_{j+1}^{(i)} } \frac {\partial u_{j+1}^{(i)} } { \partial \alpha }  \frac {\partial u_{j+1}^{(k)} } { \partial \alpha } \\
& + \sum_{i=1}^{n-1}  \frac{\partial \varphi_{j+1} } { \partial u_{j+1}^{(i)} } \frac {\partial^2 u_{j+1}^{(i)} } { \partial \alpha^2} 
+ \sum_{i=1}^{n-1}  \frac{\partial^2 \varphi_{j+1} } {\partial u_{j+1}^{(i)} \partial \alpha } \frac {\partial u_{j+1}^{(i)} } { \partial \alpha}
  -  \frac{ \partial^2 \varphi_{j}} {\partial \alpha^2 } - \sum_{k=1}^{n-1} \sum_{i=1}^{n-1}  \frac{\partial^2 \varphi_{j} } {\partial u_{j}^{(k)} \partial u_{j}^{(i)} } \frac {\partial u_{j}^{(i)} } { \partial \alpha }  \frac {\partial u_{j}^{(k)} } { \partial \alpha } \\
  &- \sum_{i=1}^{n-1}  \frac{\partial \varphi_{j} } { \partial u_{j}^{(i)} } \frac {\partial^2 u_{j}^{(i)} } { \partial \alpha^2} - \sum_{i=1}^{n-1}  \frac{\partial^2 \varphi_{j} } {\partial u_{j}^{(i)} \partial \alpha } \frac {\partial u_{j}^{(i)} } { \partial \alpha}%
\Biggl ) . \\
\end{aligned}
\end{equation*}

\noindent This is bounded by a constant $ C_d^{(2)} $ depending only on $n$ such that

\begin{equation*}
\begin{aligned}
\Big | \odv[order=2]{d_j}{\alpha} \Big| & \leq 2 \frac{(C_d^{(1)})^2 } { d_{\min}} + 2 C_\varphi^{(0,2)} + 2 (n-1)^2 C_\varphi^{(2,0)} (C_u^{(1)})^2 + 2 (n-1) C_u^{(2)} + 2 (n-1) C_\varphi^{(1,1)} C_u^{(1)} = C_d^{(2)} .
\end{aligned}
\end{equation*}

\noindent Continuing by induction, we can see that the $s'$-th derivative of $d_j(\alpha)$ is bounded by a constant which depends only on $n$ and $s'$. This proves the statement. 
\end{proof}

\bigskip

\begin{Corollary}
Let $K(\alpha)$ be a $\mathcal{C}^{r,r'}$ billiard deformation in $\mathbb{R}^n$ with $r \geq 2, r' \geq 1$. Let $v_j(\alpha)$ be the unit speed vector from $q_j(\alpha)$ to $q_{j+1}(\alpha)$. Then $v_j(\alpha)$ is $\mathcal{C}^{s'}$, where $s'=\min\{r-1,r'\}$, with respect to $\alpha$, and there exist constants $C_v^{(s')} >0$ depending only on $n$ and $s'$, such that $$\Big \| \odv[order=s']{v_j(\alpha)}{\alpha} \Big \|  \leq C^{(s')}_v.$$

\label{Dv}
\end{Corollary}

\begin{proof}

We can write $v_j(\alpha) = \frac {q_{j+1}(\alpha) - q_j (\alpha)} {d_j(\alpha)} $. 
And then by using Corollaries {\ref{Dq}} and \ref{Dd}, the statement is proved. 
\end{proof}

\bigskip

\begin{Corollary}
Let $K(\alpha)$ be a $\mathcal{C}^{r,r'}$ billiard deformation in $\mathbb{R}^n$ with $r\geq 3, r' \geq 1, n \geq 3$. Let $\nu_j(\alpha)$ be the normal vector field to the boundary of obstacle $K_{\xi_j}$ at the point $q_j(\alpha)$. Then $\nu_j(\alpha)$ is at least $\mathcal{C}^{s'}$, where $s=\min\{r-2,r'\}$, and there exist constants $C_\nu^{(s')} >0$ depending only on $n$ and $s'$ such that 
$$\Big  \| \odv[order=s']{\nu_j}{\alpha} \Big \| \leq C_\nu^{(s')}.$$
\label{Dnu1}
\end{Corollary}


\begin{proof}
%
Let $\nu_j(\alpha)$ the normal vector field to the boundary of obstacle $K_{\xi_j}$ at the point $q_j(\alpha)$. We can write
 $$ \nu_j (\alpha)= \frac{ \partial \varphi_j} {\partial u_{j}^{(1)} } \times \frac{\partial \varphi_j} {\partial u_j^{(2)} } \times \dots \times \frac{ \partial \varphi_j} {\partial u_j^{(n-1)} }.$$

\noindent Then from {{Definition \ref{44}}} and Theorem {\ref{Du}}, $\nu_j$ is at least $\mathcal{C}^{\min\{r-2,r'\}}$ with respect to $\alpha$. And to show its derivatives are bounded, we have 

\begin{equation*}
\begin{aligned}
\odv{\nu_j}{\alpha} & =\Big[ \Big ( \sum_{i=1}^{n-1} \frac{\partial^2 \varphi_{j}}{\partial u_{j}^{(i)} \partial u_j^{(1)}} \frac{ \partial u_{j}^{(i)}} { \partial \alpha} + \frac{\partial^2 \varphi_j}{\partial u^{(1)} \partial \alpha} \Big ) \times \frac{\partial \varphi_j} {\partial u_j^{(2)} } \times \ldots \times \frac{ \partial \varphi_j} {\partial u_j^{n-1} } \Big ]\\%
&+ \ldots + \Big [\frac{ \partial \varphi_j} {\partial u_{j}^{(1)} } \times \ldots \times \Big(  \sum_{i=1}^{(n-1)} \frac{\partial^2 \varphi_{j}}{\partial u_{j}^{(i)} \partial u_j^{(n-1)}} \frac{ \partial u_{j}^{(i)}} { \partial \alpha} + \frac{\partial^2 \varphi_j}{\partial u_j^{(n-1)} \partial \alpha} \Big ) \Big ]. \\
\end{aligned}
\end{equation*}

\noindent Then, there exists a constant $C_\nu^{(1)}$ depending only on $n$ such that

\begin{equation}
\begin{aligned}
\Big \| \odv{\nu_j}{\alpha} \Big \| & \leq (n-1) \Big [ (n-1) C_\varphi^{(2,0)} C_u^{(1)} + C_\varphi ^{(1,1)} \Big] = C_\nu^{(1)}.
\label{nu1}
\end{aligned}
\end{equation}

\noindent  Next, we want to show that the second derivative of $\nu_j(\alpha)$ is bounded by a constant. 

\noindent To simplify the last derivative, we can write $\odv{\nu_j}{\alpha} = \Psi_1 + \Psi_2 + \ldots + \Psi_{n-1}$, where $\Psi_i$ corresponds to one square bracket  [...]. The first derivative of $\Psi_1$ with respect to $\alpha$ is


\begin{equation*}
\begin{aligned}
\odv{\Psi_1}{\alpha} &= \Big [ \Big ( \sum_{k=1}^{n-1} \sum_{i=1}^{n-1} \frac{\partial^3 \varphi_j } {\partial u_j^k \partial u_j^{(i)} \partial u_j^{(1)} } \frac{\partial u_j^{(i)} } {\partial \alpha} \frac{\partial u_j^{(k)} } {\partial \alpha} %
+  \sum_{i=1}^{n-1} \frac{\partial^2 \varphi_{j}}{\partial u_{j}^{(i)} \partial u_j^{(1)}} \frac{ \partial^2 u_{j}^{(i)}} { \partial \alpha^2}\\
 & + 2 \sum_{i=1}^{n-1} \frac{\partial^3 \varphi_{j}}{\partial u_{j}^{(i)} \partial u_j^{(1)}\partial \alpha } \frac{ \partial u_{j}^{(i)}} { \partial \alpha} %
+ \frac{\partial^3 \varphi_j } { \partial u_j^{(1)} \partial \alpha^2} \Big ) %
 \times \ldots \times \frac{ \partial \varphi_j} {\partial u_j^{(n-1)} } \Big ] \\
 &+ \Big[ \Big ( \sum_{i=1}^{n-1} \frac{\partial^2 \varphi_{j}}{\partial u_{j}^i \partial u_j^{(1)}} \frac{ \partial u_{j}^{(i)}} { \partial \alpha} + \frac{\partial^2 \varphi_j}{\partial u^{(1)} \partial \alpha} \Big ) %
 \times \Big (  \sum_{i=1}^{n-1} \frac{\partial^2 \varphi_{j}}{\partial u_{j}^i \partial u_j^{(2)}} \frac{ \partial u_{j}^{(i)}} { \partial \alpha} + \frac{\partial^2 \varphi_j}{\partial u^{(1)} \partial \alpha} \Big )%
 \times \ldots \times  \frac{ \partial \varphi_j} {\partial u_j^{(n-1)} }\Big]\\%
&+ \ldots +   \Big[ \Big ( \sum_{i=1}^{n-1} \frac{\partial^2 \varphi_{j}}{\partial u_{j}^{(i)} \partial u_j^{(1)}} \frac{ \partial u_{j}^{(i)}} { \partial \alpha} + \frac{\partial^2 \varphi_j}{\partial u^{(1)} \partial \alpha} \Big ) %
 \times \pdv{\varphi_j}{u_j^{(2)}} \times \ldots%
 \times \Big ( \sum_{i=1}^{n-1} \frac{\partial^2 \varphi_{j}}{\partial u_{j}^i \partial u_j^{(n-1)}} \frac{ \partial u_{j}^{(i)}} { \partial \alpha} + \frac{\partial^2 \varphi_j}{\partial u^{(n-1)} \partial \alpha} \Big ) \Big ].%
\end{aligned}
\end{equation*}

\noindent And

\begin{equation*}
\begin{aligned}
\Big \| \odv{\Psi_1}{\alpha} \Big \| & \leq  (n-1)^2 C_\varphi ^{(3,0)} (C_u^{(1)})^2 + (n-1) C_\varphi^{(2,0)} C_u^{(2)} + 2 (n-1) C_\varphi ^{(2,1)} C_u^{(1)} + C_\varphi^{(1,2)} \\
& \, \, \, \, \, \, \, + (n-1) \Big[ (n-1) C_\varphi^{(2,0)} C_u^{(1)} + C_\varphi^{(1,1)} \Big ]^2 = C_\Psi.
\end{aligned}
\end{equation*}

\noindent We can get similar estimates for $\Big \| \odv{\Psi_2}{\alpha} \Big \|, \ldots, \Big \| \odv{\Psi_{n-1}}{\alpha} \Big \| $. Therefore, there exists a constant $C_\nu^{(2)} >0$ depends only on $n$ such that

\begin{equation*}
\begin{aligned}
\Big \| \odv[order=2]{\nu_j}{\alpha} \Big \| & \leq (n-1) C_\Psi = C_\nu^{(2)}. 
\end{aligned}
\end{equation*}

\noindent So, we can see by induction that the $s'$-th derivative, where $s'= \min\{r-2,r'\}$, is bounded by a constant $C_\nu^{(s')} >0$ which depends only on $n$ and $s'$. This proves the statement. 
\end{proof}

\bigskip

\begin{Corollary}
Let $K(\alpha)$ be a $\mathcal{C}^{r,r'}$ billiard deformation in $\mathbb{R}^n$ with $r\geq 3, r' \geq 1, n \geq 3 $. Let $\phi_j(\alpha)$ the collision angle at $q_j(\alpha)$. Then $\cos \phi_j(\alpha)$ is $\mathcal{C}^{s'}$, where $s'=\min\{r-2,r'\}$, and there exist constants $C_\phi^{(s')} >0$ depending only on $n$ and $s'$ such that 
$$\Big  | \odv[order=s']{\cos \phi_j}{\alpha} \Big | \leq C_\phi^{(s')}.$$
\label{Dphi}
\end{Corollary}

\begin{proof}
Recall that $ \cos \phi_j(\alpha) = \langle v_j(\alpha), \nu_j(\alpha) \rangle$. Then by using Corollaries \ref{Dv} and \ref{Dnu1}, the statement is proved. 
\end{proof}






\begin{Corollary}
Let $K(\alpha)$ be a $\mathcal{C}^{r,r'}$ billiard deformation in $\mathbb{R}^n$ with $r\geq 4, r' \geq 1, n \geq 3$. Let $N_j(\alpha)$ be the shape operator of $\partial K_{\xi_j}$ at the point $q_j(\alpha)$. Then $N_j(\alpha)$ is $\mathcal{C}^{s'}$, where $s'=\min\{r-3,r'\}$, and there exist constants $C_\kappa^{(s')} >0$ which depend only on $s'$ and $n$ such that 
$$\Big  \| \odv[order=s']{N_j}{\alpha} \Big \| \leq C_N^{(s')}.$$
\label{DN}
\end{Corollary}

\begin{proof}


The shape operator $N_j : T(\partial K_{\xi_j}( \alpha)) \to T(\partial K_{\xi_j}(\alpha))$ at the point $q_j(\alpha)$, is defined by  $N_j = \nabla \nu_j$, where $\nu_j$ is the normal vector field to the boundary of obstacle $K_{\xi_j}$ at the point $q_j(\alpha)$. 
First, from the expression of $\nu_j$, we have 

\begin{equation}
\begin{aligned}[b]
\nabla \nu_j &= \nabla \Big ( \frac{ \partial \varphi_j} {\partial u_{j}^{(1)} } \times \frac{\partial \varphi_j} {\partial u_j^{(2)} } \times ... \times \frac{ \partial \varphi_j} {\partial u_j^{(n-1)} }\Big ).\\
&= \Big [\Big (\nabla \frac{ \partial \varphi_j} {\partial u_{j}^{(1)} } \Big)  \times \frac{\partial \varphi_j} {\partial u_j^{(2)} } \times ... \times \frac{ \partial \varphi_j} {\partial u_j^{(n-1)} } \Big ]+ .... + %
\Big [ \frac{ \partial \varphi_j} {\partial u_{j}^{(1)} } \times \frac{\partial \varphi_j} {\partial u_j^{(2)} } \times ... \times \Big ( \nabla  \frac{ \partial \varphi_j} {\partial u_j^{(n-1)} }\Big ) \Big].
\label{Dnu}
 \end{aligned}
\end{equation}

\bigskip

\noindent And

\begin{equation*}
\begin{aligned}
\nabla \frac{ \partial \varphi_j} {\partial u_{j}^{(1)} } & = \Big ( \sum_{i=1}^{n-1} \pdv{\varphi_j^{(1)} }{u_j^{(i)}, u_j^{(1)}} , ... ,  \sum_{i=1}^{n-1} \pdv{\varphi_j^{(n)} }{u_j^{(i)}, u_j^{(1)}}  \Big) 
= \Big (  \sum_{i=1}^{n-1} \pdv{\varphi_j^{(k)} }{u_j^{(i)} , u_j^{(1)}} \Big )_{k=1}^{n}. 
%
%
\end{aligned}
\end{equation*}


\noindent

\noindent This is bounded by a constant $C_{\nabla _1}^{(1)}$ depending only on $n$ such that

\begin{equation}
\begin{aligned}
\Big \|  \nabla \frac{ \partial \varphi_j} {\partial u_{j}^{(1)} }  \Big \| & \leq n (n-1) C_\varphi^{(2,0)} = C_{\nabla _1}.
\label{5.10}
\end{aligned}
\end{equation}

\noindent  Also, it is $\mathcal{C}^{\min\{r-3,r'\}}$ with respect to $\alpha$ and its first derivative is


\begin{equation}
\begin{aligned}[b]
\odv{}{\alpha} \nabla \pdv { \varphi_j} {u_{j}^{(1)} } & = \odv{ } { \alpha} \Big (  \sum_{i=1}^{n-1} \pdv{\varphi_j^{(k)} } {u_j^{(i)}, u_j^{(1)}} \Big )_{k=1}^{(n)} \\%
&= \Big ( \sum_{i=1}^{n-1} \sum_{p=1}^{n-1} \pdv{\varphi_j^{(k)}}{u_j^{(p)},  u_j^{(i)}, u_j^{(1)} } \, \pdv{u_j^{(p)}}{\alpha} + \sum_{i=1}^{n-1} \pdv {\varphi_j^{(k)}} {u_j^{(i)}, u_j^{(1)}, \alpha }  \Big )_{k=1}^{(n)}.%
%
%
\end{aligned}
\end{equation}

\noindent Thus, there exists a constant $C^{(1)}_{\nabla_1}$ depending only on $n$ such that

\begin{equation}
\begin{aligned}
\Big \|  \odv{}{\alpha} \nabla \pdv { \varphi_j} {u_{j}^{(1)} }  \Big\|  & \leq n \Big [ (n-1)^2 C_\varphi^{(3,0)} C_u^{(1)} + (n-1) C_\varphi^{(2,1)}   \Big ]= C^{(1)}_{\nabla_1}. 
\end{aligned}
\label{Dnabla}
\end{equation}

\noindent And then, by deriving (\ref{Dnu}) with respect to $\alpha$ and using the estimations in  (\ref{nu1}), (\ref{5.10}) and (\ref{Dnabla}), we get

\begin{equation}
\begin{aligned}
\Big \| \odv{N_j}{\alpha}\Big \| = \Big \| \odv{}{\alpha} \nabla \nu_j \Big \|  \leq (n-1) \Big [ C_{\nabla_1} + C^{(1)}_{\nabla_1} C_\nu^{(1)} \Big ] = C_N^{(1)},
\label{nu}
\end{aligned}
\end{equation}

\noindent which is a constant depending only on $n$. By induction, we can see that the $s'$-derivative of $N_j(\alpha)$ with respect to $\alpha$ is bounded by a constant $C_N^{(s')} >0$ depending only on $n$ and $s'$.
\end{proof}

\bigskip

In the next Corollary, we want to show that 

\begin{align}
\tilde{u}_j= \Big ( \pdv{\varphi_j^{(k)}}{u_j^{(t)}} \Big )_{k=1}^{n} \in T_{q_j}(\partial{K})
\label{tildeu}
\end{align}
  
\noindent 
is differentiable with respect to $\alpha$ and its derivative is bounded by a constant independent of $j$ and $\alpha$. 
  
\bigskip

\begin{Corollary}
Let $K(\alpha)$ be a $\mathcal{C}^{r,r'}$ billiard deformation in $\mathbb{R}^n$ with $r\geq 4, r' \geq 1$. For all $j= 1,2, ...m$, let $\tilde{u}_j(\alpha)$ as in {\em{(\ref{tildeu})}}. Then $\tilde{u}_j(\alpha)$ is $\mathcal{C}^{s'}$, where $s' =\min\{r-2,r'\}$, and there exist constants $C_{\tilde{u}}^{(s')} >0$ which are independent of $s', j$ and $\alpha$ such that 
$$\Big  \| \odv[order=s']{\tilde{u}_j}{\alpha} \Big \| \leq C_{\tilde{u}}^{(s')}.$$

\label{Dtildeu}
\end{Corollary}

\begin{proof}

Let $\tilde{u}_j(\alpha)$ be as in {{(\ref{tildeu})}}. From {{Definition \ref{44}}} and Theorem {\ref{Du}}, $\tilde{u}$ is $\mathcal{C}^{\min\{r-2,r'\}}$ with respect to $\alpha$ and its first derivative is 



\begin{equation}
\begin{aligned}
\odv{\tilde{u}_j}{\alpha} & = \odv{ } {\alpha} \Big ( \pdv{\varphi_j^{(k)}}{u_j^{(t)}} \Big )_{k=1}^{n} = \Big (   \sum_{i=1}^{n-1}  \pdv{\varphi_j^{(k)}}{u_j^{(i)}, u_j^{(t)}} \pdv {u_j^{(i)}}{\alpha} +  \pdv{\varphi_j^{(k)}}{ u_j^{(t)}, \alpha}    \Big)_{k=1}^{n}.\\
\end{aligned}
\end{equation}

\noindent This is bounded by a constant $C_{\tilde{u}}^{(1)}>0$ as follows

\begin{equation}
\begin{aligned}
\Big \| \odv{\tilde{u}_j(\alpha)}{\alpha} \Big \| & \leq n\Big [(n-1) C_\varphi ^{(2,0)} C_u^{(1)} + C_\varphi ^{(1,1)} \Big ] = C_{\tilde{u}}^{(1)}. 
\label{w}
\end{aligned}
\end{equation}

\noindent It is clear that this constant depends only on $n$. And the second derivative of $\tilde{u}_j(\alpha)$ with respect to $\alpha$ is

\begin{equation}
\begin{aligned}[b]
\odv[order=2]{\tilde{u}_j}{\alpha} &  = \Big (  \sum_{p=1}^{n-1} \sum_{i=1}^{n-1}  \pdv{\varphi_j^{(k)}}{u_j^{(p)}, u_j^{(i)}, u_j^{(t)}} \pdv {u_j^{(i)}}{\alpha} \pdv {u_j^{(p)}}{\alpha}
+  \sum_{i=1}^{n-1}\pdv{\varphi_j^{(k)}}{ u_j^{(t)}, \alpha} \pdv {u_j^{(i)}}{\alpha}\\
& \, \, \, \, \, \, \,  + \sum_{i=1}^{n-1}  \pdv{\varphi_j^{(k)}}{u_j^{(i)}, u_j^{(t)}, } \frac{ \partial^2 u_{j}^{(i)}} { \partial \alpha^2}
\sum_{p=1}^{n-1}  \pdv{\varphi_j^{(k)}}{u_j^{(p)}, u_j^{(t)}, \alpha} \pdv {u_j^{(p)}}{\alpha}
+ \pdv{\varphi_j^{(k)}}{u_j^{(t)} , \alpha}  \Big)_{k=1}^{n}.
\end{aligned}
\end{equation}

\noindent Again from {{Definition \ref{44}}} and Theorem {\ref{Du}} , there exists a constant $C_{\tilde{u}}^{(2)} >0 $ such that 

\begin{equation}
\begin{aligned}
\Big \| \odv[order=2]{\tilde{u}_j(\alpha)}{\alpha} \Big \| & \leq n (n-1) \Big [(n-1) C_\varphi ^{(3,0)} (C_u^{(1)})^2 + 2 C_\varphi ^{(2,1)} C_u^{(1)} + C_\varphi ^{(2,0)} C_u^{(2)}+  C_\varphi ^{(1,2)} \Big ] = C_{\tilde{u}}^{(2)}. 
\label{w}
\end{aligned}
\end{equation}

\noindent This constant depends only on $n$. By induction, we can see that the $s'$-th derivative of ${\tilde{u}}_j$ with respect to $\alpha$ is bounded by a constant depends only on $s'$ and $n$. This proves the statement.  
\end{proof}

The next Corollary follows from Corollaries {\ref{DN}} and {\ref{Dtildeu}}.   
  
\begin{Corollary}
Let $K(\alpha)$ be a $\mathcal{C}^{r,r'}$ billiard deformation in $\mathbb{R}^n$ with $r\geq 4, r' \geq 1$. Let $\kappa_j(\alpha)$ be the normal curvature of the boundary $\partial K_{\xi_j}$ at the point $q_j(\alpha)$ in the direction $ u_j \in T_{q_j} (\partial{K}_j)$, which is given by $\kappa_j(\alpha) = \langle N_j (\tilde{u}_j), \tilde{u_j} \rangle$ . Then $\kappa_j(\alpha)$ is $\mathcal{C}^{s'}$, where $s'=\min\{r-3,r'\}$ and there exist constants $C_\kappa^{(s')} >0$ which depend only on $s'$ and $n$ such that 
$$\Big  | \odv[order=s']{\kappa_j}{\alpha} \Big | \leq C_\kappa^{(s')}.$$
\label{Dkappa}
\end{Corollary}

\begin{Corollary}
Let $K(\alpha)$ be a $\mathcal{C}^{r,r'}$ billiard deformation in $\mathbb{R}^n$ with $r\geq 3, r' \geq 1$. Let $U_{j}$ and $U^{-1}_{j}$ be as in {\em{(\ref{U})}} and {\em{(\ref{U-1})}}. Then $U_{j}$ and $U^{-1}_{j}$ are at least $\mathcal{C}^{s'}$, where $s'=\min\{r-2,r'\}$ and there exist constants $C_{U}^{(s')} >0$ depending only on $n$ and $s'$ such that 

$$\Big  \| \odv[order=s']{U_{j}}{\alpha} \Big \| \leq C_{n }^{(s')}         \, \, \, \,   \, \, \  {{and}} \, \, \, \,  \, \, \  \Big  \| \odv[order=s']{U^{-1}_{j}}{\alpha} \Big \| \leq C_{n}^{(s')} .$$

\label{DU}
\end{Corollary}

\begin{proof}

From (\ref{U}) and (\ref{U-1}), we can write $$U_j = I - \frac{\nu_j v_{j-1}^{\intercal }} {\cos \phi_j}\, \, \, \, \, \, \,\,\, \,  {\text{and}} \, \, \, \, \, \, \,\,\, \,  U_j^{-1} = I - \frac{\nu_j v_{j+1}^{\intercal }} {\cos \phi_j}, $$ 

\noindent where the operators $U_j^{-1}$ and $U_j$ depend on $\alpha$. From Corollaries \ref{Dv}, \ref{Dnu1} and \ref{Dphi}, $U_j$ and $U_j^{-1}$ are $\mathcal{C}^{\min\{r-2, r'\}}$.  Moreover, we can see that their first derivatives are bounded by same constant $C_{n}^{(1)}$ such that, 

\begin{equation}
 C_{n}^{(1)} = \frac{C_\nu^{(1)} + C_v^{(1)} } {\cos {\phi_{\max}} },
 \label{Cn1}
\end{equation}

\noindent where $\cos \phi_{\max} \in [0, \pi/2 )$. $C_{n}^{(1)}$ is only depending on $n$. And the second derivatives of $U_j$ and $U_j^{-1}$ with respect to $\alpha$ are bounded by $C_{n}^{(2)}$ such that 

\begin{equation}
C_{n}^{(2)} = \frac{1}{(\cos {\phi_{\max}})^4 } C^{(1)}_\phi \big( 4C_\nu^{(1)} + 4C_v^{(1)} + C_\phi^{(1)} \big) + C_\nu^{(2)} + C_\nu^{(1)}C_v^{(1)} + C_v^{(2)} + C^{(2)}_\phi.
\label{Cn2}
\end{equation}

\noindent This constant depends only on $n$. By induction, we can see that the $s'$-th derivatives of $U_j$ and $U_j^{-1}$ are bounded by a constant $C_{n}^{(s')} >0$ that depends only on $n$ and $s'$. This proves the statement. 
\end{proof}

\begin{Corollary}
Let $K(\alpha)$ be a $\mathcal{C}^{r,r'}$ billiard deformation in $\mathbb{R}^n$ with $r\geq 3, r' \geq 1$. Let ${V}_{j}^*$ and ${V}_{j}$ be as in {\em{(\ref{V})}} and {\em{(\ref{V*})}}. Then ${V}_{j}^*$ and ${V}_{j}$ are at least $\mathcal{C}^{s'}$, where $s'=\min\{r-2,r'\}$ and there exist constants $C_{{n}}^{(s')} >0$ depending only on $n$ and $s'$ such that 

$$\Big  \| \odv[order=s']{{V}_{j}^*}{\alpha} \Big \| \leq C_{{n}}^{(s')}         \, \, \, \,   \, \, \  {{and}} \, \, \, \,  \, \, \  \Big  \| \odv[order=s']{{V}_{j}}{\alpha} \Big \| \leq C_{{n}}^{(s')} .$$
\label{DV*}
\end{Corollary}

\begin{proof}

From (\ref{V}) and (\ref{V*}), we can write $$V = I - \frac{ v_j \nu_j^{\intercal}} {\cos \phi_j} \, \, \, \, \, {\text{and}}\, \, \, \, \,  V^* = I - \frac{ \nu_j v_j^{\intercal}} {\cos \phi_j}.$$ Then, the rest of the proof is the same as the proof of Corollary \ref{DU}. 
\end{proof}

\bigskip

The next Corollary follows from Corollaries {\ref{Dphi}},  {\ref{DV*}} and {\ref{DN}}. 

\begin{Corollary}
Let $K(\alpha)$ be a $\mathcal{C}^{r,r'}$ billiard deformation in $\mathbb{R}^n$ with $r\geq 4, r' \geq 1$. Let $\Theta_j(\alpha) = {\cos\phi_j} {V}_{j}^* N_{j} {V}_{j}$. Then $\Theta_j(\alpha)$ is $\mathcal{C}^{s'}$, where $s'=\min\{r-3,r'\}$ and there exist constants $ C_{\Theta}^{(s')} >0$ depending only on $s'$ and $n$ such that 

$$\Big  \| \odv[order=s']{\Theta_j}{\alpha} \Big \| \leq C_{\Theta}^{(s')}    .$$
\label{Dtheta}
\end{Corollary}

\bigskip

In the next Corollary, we want to show that $\hat{w}_j(\alpha) \in T_{q_j}(X_j(\alpha))$, such that \newline $\hat{w}_j = w_j / \| w_j \|$ and
\begin{equation}
w_j= \tilde{u}_j - \frac {\langle \tilde{u}_j , v_j \rangle } {\cos \phi_j} \nu_j,
\label{W}
\end{equation}

\noindent where $\tilde{u}_j$ is as in (\ref{tildeu}).

\begin{Corollary}
Let $K(\alpha)$ be a $\mathcal{C}^{r,r'}$ billiard deformation in $\mathbb{R}^n$ with $r\geq 4, r' \geq 1, n \geq 3$. 
For all $j= 1,2, ...m$, let ${w}_j(\alpha)$ be as in {\em(\ref{W})}. Then $\hat{w}_j(\alpha)$ is $\mathcal{C}^{s'}$, where $s' =\min\{r-2,r'\}$, and there exist constants $C_{\hat{w}}^{(s')} >0$ which are independent of $s', j$ and $\alpha$ such that 
$$\Big  \| \odv[order=s']{\hat{w}_j}{\alpha} \Big \| \leq C_{\hat{w}}^{(s')}.$$

\label{hatw}
\end{Corollary}

\begin{proof}

Let $w_j(\alpha) \in \mathcal{J}_j$ as in (\ref{W}), where $\mathcal{J}_j$ is the hyperplane to the convex front $\widetilde{X}_j$ at the point $q_j(\alpha)$. 
From Corollaries \ref{Dv}, \ref{Dnu1} and \ref{Dtildeu}, $\hat{w}_j$ is at least $\mathcal{C}^{\min \{r-2, r'\}}$ with respect to $\alpha$, and there exists a constant $C_{\hat{w}}^{(1)} >0$ such that 

$$\Big \| \odv{\hat{w}_j}{\alpha} \Big \| \leq \frac{2}{(1+ \cos \phi^+) \cos^2\phi^+ } \Big (2 C_{\tilde{u}}^{(1)} + C_v^{(1)} + C_\nu^{(1)} + C_\phi^{(1)} \Big ) = C_{\hat{w}}^{(1)} .$$

\noindent This constant depends only on $n$. By induction, we can see that the $s'$-th derivative of $\hat{w}_j$ with respect to $\alpha$ is bounded by a constant depends only on $s'$ and $n$. This proves the statement.  
\end{proof}

\begin{Corollary}
Let $K(\alpha)$ be a $\mathcal{C}^{r,r'}$ billiard deformation in $\mathbb{R}^n$ with $r\geq 4, r' \geq 1$. 
Then for all $j =1,2,..., m$, $\mathcal{K}_j(\alpha)$ is $\mathcal{C}^{s'}$, where $s=\min\{r-3,r'\}$ and there exists a constant $C_{\mathcal{K}}^{(s')} >0$ which depends on $s'$ and $n$ such that 
$$\Big  \| \odv[order=s']{\mathcal{K}_j(\alpha)}{\alpha} \Big \| \leq C_{\mathcal{K}}^{(s')}.$$
\label{DKK2}
\end{Corollary}

\begin{proof}

From (\ref{K2}), recall that 

$$\mathcal{K}_{j} (\alpha) = U_j^{-1} \mathcal{K}^{-}_j U_j + 2 \Theta_{j} ,$$

\noindent where $\mathcal{K}^{-}_{j} (\alpha) = \mathcal{K}_{j-1}(I +d_{j-1}\mathcal{K}_{j-1} )^{-1}$. 
Given that every term in the right-hand side of the equation is differentiable with respect to $\alpha$, and we can derive $\mathcal{K}_{j-1}$ in terms of $\mathcal{K}_0$ on the right-hand side, where $\mathcal{K}_0$ is independent of $\alpha$, we can conclude that $\mathcal{K}_j$ is differentiable with respect to $\alpha$. Specifically, $\mathcal{K}_j$ is at least $\mathcal{C}^{\min \{r-3,r'\} }$, where $r\geq 4$ and $r' \geq 1$. This result follows from the Corollaries \ref{Dtheta} and \ref{DU}.

\bigskip

 Next, we want to show that the first dervative of $\mathcal{K}_j$ with respect to $\alpha$ is bounded by a constant independent of $\alpha$ and $j$.  First, we have

\begin{equation}
\begin{aligned}[b]
\odv{\mathcal{K}^{-}_{j}}{\alpha} &=  \odv{\mathcal{K}_{{j-1}}}{\alpha}  (I + d_{j-1} \mathcal{K}_{j-1})^{-1} -(I + d_{j-1} \mathcal{K}_{j-1})^{-1} ( d_{j-1} \odv{\mathcal{K}_{j-1}}{\alpha}  + \odv{d_{j-1}}{\alpha} \mathcal{K}_{j-1}) (I + d_{j-1} \mathcal{K}_{j-1})^{-1} \\
&=  \Big [ \Big ( I - d_{j-1} \mathcal{K}_{j-1} (I + d_{j-1} \mathcal{K}_{j-1})^{-1} \Big ) \odv{\mathcal{K}_{j-1}}{\alpha} \Big ]  (I + d_{j-1} \mathcal{K}_{j-1})^{-1}\\ 
& \, \, \, \, \, \, \, \, \, \,   -(I + d_{j-1} \mathcal{K}_{j-1})^{-1} \odv{d_{j-1}}{\alpha} \mathcal{K}^{2}_{j-1} (I + d_{j-1} \mathcal{K}_{j-1})^{-1} \\ 
&= (I + d_{j-1} \mathcal{K}_{j-1})^{-1}  \odv{\mathcal{K}_{j-1}}{\alpha}  (I + d_{j-1} \mathcal{K}_{j-1})^{-1} -(I + d_{j-1} \mathcal{K}_{j-1})^{-1} \odv{d_{j-1}}{\alpha} \mathcal{K}^{2}_{j-1} (I + d_{j-1} \mathcal{K}_{j-1})^{-1}\\ 
& = \mathcal{D}_{j-1} \odv{\mathcal{K}_{j-1}}{\alpha} \mathcal{D}_{j-1} + \mathcal{E}_{{j-1}}
\label{k-1}
\end{aligned}
\end{equation}

\noindent where $\mathcal{D}_{j-1}(\alpha) = (I + d_{j-1} \mathcal{K}_{j-1})^{-1}$ and $\mathcal{E}_{{j-1}} =  - \odv{d_{j-1}}{\alpha}  \mathcal{D}_{j-1} \mathcal{K}^2_{j-1}  \mathcal{D}_{j-1}$ .  Second, we derive 

\begin{align*}
\odv{\mathcal{K} _j}{\alpha} & = U_{j}^{-1} \odv{\mathcal{K}^{-}_{j-1}}{\alpha} U_{j} - U_{j}^{-1} \odv{U_{j}^{-1}}{\alpha} U_{j}^{-1} \mathcal{K}_{j-1} U_j + U_{j}^{-1} \mathcal{K}^{-}_{j-1} \odv{U_{j}}{\alpha} + 2 \odv{\Theta_{j}}{\alpha}.
\end{align*}

\noindent From (\ref{k-1}), we get

\begin{align*}
\odv{\mathcal{K} _j}{\alpha} & =  U^{-1}_j \mathcal{D}_{j-1}\odv{\mathcal{K}_{j-1}}{\alpha} \mathcal{D}_{j-1} U_j + U^{-1}_j \mathcal{E}_{j-1} U_j -  \odv{U_{j-1}^{-1}}{\alpha} \mathcal{K}_{j-1}  +  \mathcal{K}^{-}_{j-1} \odv{U_{j-1}}{\alpha} + 2 \odv{\Theta_{j}}{\alpha}.
\end{align*}

\noindent Let $\widetilde{\mathcal{P}}_j = U^{-1}_j \mathcal{D}_{j-1}$,  $\mathcal{P}_j = \mathcal{D}_{j-1} U_j $, and $\mathcal{R}_j =  U^{-1}_j \mathcal{E}_{{j-1}} U_j -  \odv{U_{j}^{-1}}{\alpha} \mathcal{K}_{j-1}  +  \mathcal{K}^{-}_{j-1} \odv{U_{j}}{\alpha} + 2 \odv{\Theta_{j}}{\alpha}$. So, we can write

\begin{align*}
\odv{\mathcal{K} _j}{\alpha} & =  \widetilde{\mathcal{P}}_{j} \odv{\mathcal{K}_{j-1}}{\alpha} \mathcal{P}_{j} + \mathcal{R}_j.
\end{align*}


\noindent To get the estimation, we have

\begin{align*}
\odv{\mathcal{K} _j}{\alpha} & 
%
= \widetilde{\mathcal{P}}_{j} \Big [ \widetilde{\mathcal{P}}_{j-1} \odv{\mathcal{K}_{j-2}}{\alpha} \mathcal{P}_{j-1} + \mathcal{R}_{j-1} \Big ] \mathcal{P}_{j} + \mathcal{R}_{j}\\
& = \widetilde{\mathcal{P}}_{j} \Big [ \widetilde{\mathcal{P}}_{j-1} \Big ( \widetilde{\mathcal{P}}_{j-2} \odv{\mathcal{K}_{j-3}}{\alpha} \mathcal{P}_{j-2} + \mathcal{R}_{j-2} \Big ) \mathcal{P}_{j} + \mathcal{R}_{j-1} \Big ] \mathcal{P}_{j-1} + \mathcal{R}_{j}\\
&=\widetilde{\mathcal{P}}_{j} \widetilde{\mathcal{P}}_{j-1} \, \ldots\, \widetilde{\mathcal{P}}_1 \odv{\mathcal{K}_{0}}{\alpha} \mathcal{P}_1 \, \ldots\, \mathcal{P}_{j-1} \mathcal{P}_{j} + \widetilde{\mathcal{P}}_{j}\mathcal{R}_{j-1} \mathcal{P}_{j}+ \ldots+ \widetilde{\mathcal{P}}_{j} \widetilde{\mathcal{P}}_{j-1}  \ldots \widetilde{\mathcal{P}}_2 \mathcal{R}_{1} \mathcal{P}_2 \, \ldots \, \mathcal{P}_{j-1} \mathcal{P}_{j} +  \mathcal{R}_{j}.\\
\end{align*}

\noindent Since $\mathcal{K}_{0}$ is independent of $\alpha$
, then 

\begin{align*}
\odv{\mathcal{K} _j}{\alpha} =   \widetilde{\mathcal{P}}_{j}\mathcal{R}_{j-1} \mathcal{P}_{j}+ \ldots+ \widetilde{\mathcal{P}}_{j} \, \ldots \,  \widetilde{\mathcal{P}}_2 \mathcal{R}_{1} \mathcal{P}_2 \, \ldots \, \mathcal{P}_{j} +  \mathcal{R}_{j}.
\end{align*}

From (\ref{eta}) and Corollaries \ref{Dd}, \ref{Dtheta} and \ref{DU},  we have $\| \mathcal{P}_j \| \leq \mathcal{D}_{\max}$ and $\| \widetilde{\mathcal{P}}_j \| \leq \mathcal{D}_{\max}$ , where $\mathcal{D}_{\max} = \frac{1}{1+d_{\min} \mu_{\min}}$, and  

$$\|\mathcal{R}_j \| \leq \mathcal{R}_{\max} = (\mathcal{D}_{\max})^2 C_d^{(1)} (\eta_{\max})^2 + C_w^{(1)} \eta_{\max} + \frac{C_w^{(1)} }{d_{\min} }+2 C_{\Theta}^{(1)}.$$

\noindent And then

\begin{align*}
\Big \| \odv{\mathcal{K}^{}_{j}}{\alpha} \Big \| & \leq   \Big [ 1 + (\mathcal{D}_{\max} )^2 + (\mathcal{D}_{\max} )^4 + ... + (\mathcal{D}_{\max} )^{2(j-1)} \Big ] \mathcal{R}_{\max}  \\
& =  \frac {(1 - (\mathcal{D}_{\max} )^{2j} }{1-({\mathcal{D}_{\max}})^{2}} \mathcal{R}_{\max} .   
\end{align*}

\noindent For a large $j$, $(\mathcal{D}_{\max} )^{2j}$ converges to $0$. Thus, there exists a constant $C_{\mathcal{K}^{}}^{(1)} > 0$ independent of $j$ and $\alpha$, such that $$\Big \| \odv{\mathcal{K}^{}_{j}}{\alpha} \Big \|  \leq C_{\mathcal{K}^{}}^{(1)}.$$

By using the same approach we can see that the $s'$-th derivative of $\mathcal{K}_j$ is bounded by a constant $C_{\mathcal{K}^{}}^{(s')} > 0$, which depends only on $n$ and $s'$. 
\end{proof}

\bigskip

The next Corollary follows from Corollaries  \ref{hatw} and \ref{DKK2}.   
  
\begin{Corollary}
Let $K(\alpha)$ be a $\mathcal{C}^{r,r'}$ billiard deformation in $\mathbb{R}^n$ with $r\geq 4, r' \geq 1$. Let $k_j(\alpha)$ be the normal curvature of the unstable manifold $\widetilde{X}_j$ at the point $q_j(\alpha)$ in the direction $ \hat{w}_j \in T_{q_j} (\widetilde{X}_j)$, which is given by $k_j(\alpha) = \langle \mathcal{K}_j (\hat{w}_j), \hat{w}_j \rangle$ . Then $k_j(\alpha)$ is $\mathcal{C}^{s'}$, where $s'=\min\{r-3,r'\}$ and there exist constants $C_k^{(s')} >0$ which depend only on $s'$ and $n$ such that 
$$\Big  | \odv[order=s']{k_j}{\alpha} \Big | \leq C_k^{(s')}.$$
\label{Dkappa}
\end{Corollary}


\begin{Corollary}
Let $K(\alpha)$ be a $\mathcal{C}^{r,r'}$ billiard deformation in $\mathbb{R}^n$ with $r\geq 4, r' \geq 1$. Let $\ell_j(\alpha)$ be as defined in {\em{(\ref{L 2})}}. Then $\ell_j(\alpha)$ is $\mathcal{C}^{s'=\min\{r-3,r'\}}$ and there exist constants $C_{\ell}^{(s')} >0$ which depend on $s'$ and $n$ such that 
$$\Big  | \odv[order=s']{\ell_j}{\alpha} \Big | \leq C_{\ell}^{(s')}.$$
\label{Dell}
\end{Corollary}

\begin{proof}
Recall that 
$$  (1+ d_j(\alpha) \ell_j(\alpha))^2 = \| \hat{w}_j(\alpha) + (d_j (\alpha)) \mathcal{K}_j (\hat {w}_j(\alpha)) \|^2.$$


\noindent 

\noindent From this and Corollaries \ref{Dd}, \ref{hatw} and \ref{DKK2}, $\ell_j(\alpha)$ is $\mathcal{C}^{s'}$, where $s' = \min\{r-3,r'\}$. And its first derivative with respect to $\alpha$ is 

\begin{equation*}
\begin{aligned} 
%
%
\odv{\ell_j}{\alpha} & = \frac{1}{d_j} \Big [ \frac { \hat{w}_j + d_j \mathcal{K}_j (\hat{w}_j)} { \| \hat{w}_j + d_j \mathcal{K}_j(\hat{w}_j) \|} ( \odv{}{\alpha}\hat{w}_j + \odv{}{\alpha}{d_j} \mathcal{K}_j (\hat{w}_j)+ d_j \odv{}{\alpha}{\mathcal{K}}_j(\hat{w}_j)  + d_j \mathcal{K} \odv{}{\alpha} \hat{w}_j - \odv{}{\alpha}{d}_j \ell_j \Big ].
\end{aligned}
\end{equation*}

\noindent And again by Corollaries \ref{Dd}, \ref{hatw} and \ref{DKK2}, we have 

\begin{equation*}
\begin{aligned} 
\Big | \odv{\ell_j}{\alpha} \Big | \leq \frac{1}{d_{\min}} \big( C_{w_1}^{(1)} + 2 C_d^{(1)} \eta_{\max} + d_{\max} C_{\mathcal{K}_1}^{(1)} \big ).
\end{aligned}
\end{equation*}

\noindent Then, there exists a $C_{\ell_1}^{(1)} >0 $ depending on $n$, such that $\Big | \odv{\ell_j}{\alpha} \Big | \leq C_{\ell_1}^{(1)}$. By induction, we can see that $\Big | \odv[order=s']{\ell_j}{\alpha} \Big | \leq C_{\ell_j}^{(s')}$, where $ C_{\ell_j}^{(s')}$ is constant depending only on $s'$ and $n$.  %
 \end{proof}



\section{Estimation the largest Lyapunov Exponent for open billiard deformation}

\label{1L}

For every $\alpha \in [0,b]$, let $M_\alpha$ be the non-wandering set for the billiard map related to $K(\alpha)$
and let $R_\alpha : M_\alpha \longrightarrow \Sigma$ be the analogue of the conjugacy map $R: M_0\longrightarrow \Sigma$,
so that the following diagram  is commutative:
$$\def\normalbaselines{\baselineskip20pt\lineskip3pt \lineskiplimit3pt}
\def\mapright#1{\smash{\mathop{\longrightarrow}\limits^{#1}}}
\def\mapdown#1{\Big\downarrow\rlap{$\vcenter{\hbox{$\scriptstyle#1$}}$}}
\begin{matrix}
M_\alpha &\mapright{B_{\alpha}}& M_\alpha\cr 
\mapdown{R_\alpha}& & \mapdown{R_\alpha}\cr \Sigma &\mapright{\sigma}& \Sigma
\end{matrix}
$$ 
where $B_\alpha$ is the billiard ball map on $M_\alpha$. By Theorem \ref{OSE} there exists a subset $A_\alpha$ of $\Sigma$
with $\mu(A_\alpha) = 1$ so that 
\begin{equation}
\lambda_1(\alpha) = \lim_{m\to\infty} \frac{1}{m} \log \| D_{x_0}B^m_\alpha(w)\|
\label{B_alpha}
\end{equation}
for all $x \in M_\alpha$ with $R_\alpha(x) \in A_\alpha$, and all $ w \in E^u_1(x) \oplus E^u_2(x) \oplus \dots E^u_{k}(x) \backslash E^u_2(x) \oplus E^u_3(x) \oplus \dots E^u_{k}(x)$. Similarly, let $A_0$ be the set with $\mu(A_0) = 1$
which we get from Theorem \ref{OSE} for $\alpha = 0$.

\bigskip

\begin{Lemma}
Given an arbitrary sequence 
\begin{equation*}
\alpha_1, \alpha_2, \ldots, \alpha_{p}, \ldots 
\end{equation*} 
of elements of $[0,b]$, for 
$\mu$-almost all $\xi\in \Sigma$ the formula {\rm{(\ref{B_alpha})}} is valid for $\alpha = \alpha_p$ and $x = R_{\alpha}^{-1}(\xi)$ for all $p = 1,2, \ldots$
and also for $\alpha = 0$ and $x = R^{-1}(\xi)$.

\label{alpha_tilde{k}}

\end{Lemma}

\begin{proof}
See \cite{Amal}.
\end{proof}

\bigskip 
Thus, using the notation $x(0,\alpha) \in M_\alpha$, we can choose $\xi \in \Sigma$ so that 
formula (\ref{B_alpha}) applies for $\alpha = \alpha_{p}$ and $x = x (0,\alpha_{p})$ for all ${p} = 1, 2, \ldots$, and also for $\alpha = 0$
and $x = (0,0)$.

\bigskip

From the formula for the largest Lyapunov exponent (\ref{g}), we can write the Lyapunov exponents for $K(\alpha)$ and $K(0)$ as follows: 

\begin{equation*}
\begin{aligned}
\lambda_1 (\alpha)&=\lim_{m\to \infty}  \frac{1}{m} \sum_{j=1}^{m} \log {\Big(1+d_j (\alpha) \ell_j (\alpha)\Big)}= \lim_{m\to \infty} \lambda^{(m)}_1 (\alpha), \\
 \lambda_1 (0)&=  \lim_{m\to \infty}  \frac{1}{m} \sum_{j=1}^{m} \log {\Big(1+d_j (0) \ell_j (0)\Big)}= \lim_{m\to \infty} \lambda^{(m)}_1 (0),
\end{aligned}
\end{equation*}

\noindent where

\begin{equation}
\begin{aligned}
\lambda^{(m)}_1 (\alpha)=  \frac{1}{m} \sum_{j=1}^{m} \log {\Big(1+d_j (\alpha) \ell_j (\alpha)\Big)}  \\
{\rm{and}} \, \, \,\, \,  
\lambda^{(m)}_1 (0)=   \frac{1}{m} \sum_{j=1}^{m} \log {\Big(1+d_j (0) \ell_j (0)\Big)}.\\
\label{lambda_k}
\end{aligned}
\end{equation}

\bigskip

\section{{\large{The continuity of the largest Lyapunov exponent for non-planar billiard deformation}}}

\label{ContL}

This section provides a rigorous proof that the largest Lyapunov exponent of the billiard deformation is a continuously varying function with respect to the perturbation parameter $\alpha$ in $\mathbb{R}^n$ with $n \geq 3$. Our proof closely follows the argument used in the proof of the continuity of the largest Lyapunov exponent for planar billiard deformation \cite{Amal}, as well as the results established in Section \ref{derivative} and previous sections.

\bigskip 

Here we prove Theorem \ref{cont2.}:








\proofn[of Theorem \ref{cont2.} ]
We will prove that $\lambda_1(\alpha)$ is continuous at $\alpha = 0$. From this continuity at every $\alpha \in [0, b]$ follows. To prove that, we will show that $\lambda^m_1(\alpha_p) \to \lambda^m_1(0)$ for any sequence of points  $\alpha_p \in [0,b]$ as in Lemma \ref{alpha_tilde{k}} with $\alpha_p \to 0$ when $p \to \infty$.
From (\ref{lambda_k}), we have



\begin{align*}
 \lambda_1 ^ m (\alpha_p) - \lambda_1 ^ m (0)   &= 
 - \frac{1}{m} \sum_{i=1}^{m} \Big [ \log (1+ d_i(\alpha_p) \ell_i(\alpha_p)) - \log(1+ d_i(0) \ell_i(0)) \Big ]. 
%
\end{align*}

\noindent Therefore, 

\begin{align*}%
\Bigg | \lambda_1 ^ m (\alpha_p) - \lambda_1 ^ m (0) \Bigg | &\leq  \frac{1}{m} \sum_{i=1}^{m}  \Bigg |  \frac {1+ d_i(\alpha_p) \ell_i(\alpha_p)- (1+ d_i (0) \ell_i(0)) } {1+ \min\{d_i (\alpha_p) \ell_i (\alpha_p), d_i (0) \ell_i (0)\}} \Bigg | \\
&=  \frac{1}{m} \sum_{i=1}^{m}  \Bigg |  \frac { d_i(\alpha_p) \ell_i(\alpha_p)-  d_i(0)  \ell_i (0)} {1+ d_{\min}  \mu_{\min} } \Bigg | \\
&=  \frac{1}{m} \, \mathcal{D}_{\min} \sum_{i=1}^{m}  \Bigg | d_i(\alpha_p) \ell_i(\alpha_p)- d_i (0) \ell_i (0)\Bigg | \\
&=  \frac{1}{m} \, \mathcal{D}_{\min} \sum_{i=1}^{m}  \Bigg | \Big (d_i(\alpha_p) - d_i (0) \Big ) \ell_i(\alpha_p ) + d_i(0) \Big (\ell_i (\alpha_p ) - \ell_i(0) \Big ) \Bigg |, 
\end{align*}

\bigskip

\noindent where $\mathcal{D}_{\min} = \frac{1}{1+d_{min} \mu_{min}} >0$ is a global constant independent of $\alpha_p$. 

\bigskip

\noindent From the Mean Value Theorem,  we have $| \ell _i (\alpha_p) - \ell_i(0)| = \alpha_p | \dot{\ell}_i(s(\alpha_p))|$ and $| d_i (\alpha) - d_i(0)| = \alpha_p |\dot{d}_i(t(\alpha_p))|$, where $\dot{\ell}_i$ and $\dot{d}_i$ are the derivatives of $\ell_i$ and $d_i$ with respect to $\alpha_p$, for some $s(\alpha_p), t(\alpha_p) \in [0,\alpha_p] $. From Corollary \ref{Dd}, there exists a constant  and $C^{(1)}_d$ such that $|\dot{d}_i(s(\alpha))|  \leq C^{(1)}_d$. And from Corollary \ref{Dell}, there exists a constant $C^{(1)}_{\ell_i}$ such that $|\dot{\ell}_i(s(\alpha))|  \leq C^{(1)}_{\ell}$ independent of $i$. 
Therefore for all $i$, $| \ell_i (\alpha_p) - \ell_i(0)|  \leq  \alpha_p  C^{(1)}_{\ell}$, and $| d_i (\alpha_p) - d_i(0)| \leq \alpha_p C^{(1)}_d  $. Then

\begin{equation*}
\begin{aligned}
 \Big| \lambda_1 ^ m (\alpha_p) - \lambda_1 ^ m (0) \Big |  &\leq \frac{1}{m} \, \mathcal{D}_{\min} \sum_{i=0}^{m-1}  \Big(  \Big| {d_i (\alpha_p) - d_i (0)\Big|} k_i (\alpha_p) + d_i (0)  \Big| {\ell_i (\alpha_p) - \ell_i(0)}   \Big| \Big) \\
&\leq  \frac{1}{m} \alpha_p \, \mathcal{D}_{\min} \sum_{i=1}^{m} (C^{(1)}_d \eta_{\max} + C^{(1)}_\ell d_{\max} ) \\
&= \alpha_p \mathcal{D}_{\min} \Big  (C^{(1)}_d \eta_{\max} + C^{(1)}_\ell d_{\max} \Big )\\
\end{aligned}
\end{equation*}

\bigskip

\noindent By letting $\alpha_p$ approach $0$ as $p$ approaches infinity, and then letting $m$ approach infinity, we get  $\lambda_1 (\alpha_p)$ approaching $\lambda_1  (0)$. This proves the statement.\qed

\bigskip


\section{{\large{The differentiability of the largest Lyapunov exponent for the non-planar billiard deformation}}}

\label{DiffL}

 
This section presents a rigorous proof that the largest Lyapunov exponent is differentiable with respect to a small perturbation $\alpha$ in $\mathbb{R}^n$ with $n \geq 3$. The proof is very similar to the proof in \cite{Amal} for the case of planar billiard deformation, and builds on the results established in Section \ref{derivative} and previous sections. Our results demonstrate that the largest Lyapunov exponent of the billiard deformation is a differentiable function of the perturbation parameter $\alpha$ .


\bigskip 

Here is the proof of  Theorem \ref{diff2.}. We will use the same argument in the proof of the differentiability of the largest Lyapunov exponent in the planar open billiards in \cite{Amal}.

\bigskip

\proofn[of Theorem \ref{diff2.}]
We will prove differentiability at $\alpha=0$. From this differentiability at any $\alpha \in [0, b]$ follows. To prove the differentiability at $\alpha =0$, we have to show that there exists 
\begin{equation*}
\begin{aligned}
\lim_{\alpha \to 0}\frac {\lambda_1 (\alpha) - \lambda_1 (0) } {\alpha}.
\end{aligned}
\end{equation*}

\noindent Equivalently, there exists a number $F$ such that 

\begin{equation*}
\begin{aligned}
\lim_{p \to \infty}\frac {\lambda_1 (\alpha_p) - \lambda_1 (0) } {\alpha_p} = F,
\end{aligned}
\end{equation*}

\noindent for any sequence $\alpha_1 > \alpha_2 > ... > \alpha_{p} > ... \to 0$ as $p \to \infty $ in $[0, b]$ as in Lemma \ref{alpha_tilde{k}}.

 \bigskip 
 

 \noindent By using Lemma \ref{alpha_tilde{k}} and the expressions of $\lambda_1^{m}(\alpha)$ for $\alpha = \alpha_p$ and $\lambda_1^{m}(0)$ in (\ref{lambda_k}), we have $ \lambda_1 ^ {(m)} (\alpha_p) \to \lambda_1  (\alpha_p)$ and $ \lambda_1 ^ {(m)} (0) \to \lambda_1  (0)$ when $m \to \infty$. 

 %
 %

\noindent Set $f_j (\alpha_p) = \log \big (1+d_j (\alpha_p) \ell_j(\alpha_p) \big)$ and $f_j (0) = \log \big (1+d_j (0) \ell_j(0) \big)$. As in \cite{Amal}, we derive

\begin{equation*}
\begin{aligned}
 \frac{ \lambda_1 ^ {(m)} (\alpha_p) - \lambda_1 ^ {(m)} (0)}{\alpha_p} &= - \frac{1}{m} \sum_{j=1}^{m}\frac{ f_j (\alpha_p)- f_j (0)}{\alpha_p}.  \\
\end{aligned}
\end{equation*}

\noindent And Taylor's formula gives

\begin{equation*}
\begin{aligned}
f_j(\alpha_p) &= f_j(0) + \alpha_p \dot{f}_j (0) + \frac{\alpha_p^2}{2} \ddot{f}_j(r_j(\alpha_p))
\end{aligned}
\end{equation*}
 
 \noindent for some $r_j(\alpha_p) \in [0, \alpha_p]$. Then 
 \begin{equation*}
\begin{aligned}
\frac{f_j(\alpha_p) - f_j(0)}{\alpha_p} - \dot{f}_j (0) &= \frac{\alpha_p}{2} \ddot{f}_j(r_j(\alpha_p)).
\end{aligned}
\end{equation*}
 
\noindent Let \, $$ F_m = \frac{1}{m} \sum_{j=1}^{m} \dot{f}_j(0).$$ 

\noindent Summing up the above for $j=1,2,..., m$, we get

\begin{equation*}
\begin{aligned}
\frac{ \lambda_1 ^ {(m)} (\alpha_p) - \lambda_1 ^ {(m)} (0)}{\alpha_p}  - F_m &= - \frac{1}{m} \sum_{j=1}^{m} \Big[\frac{ f_j (\alpha_p)- f_j (0)}{\alpha_p} - \dot{f}_j (0) \Big]. \\
\end{aligned}
\end{equation*}

\noindent From the definition of $f_j(\alpha_p)$, 

\begin{equation*}
\begin{aligned}
\dot{f}_j(\alpha_p) &= \frac{\dot{d}_j(\alpha_p) \ell_j(\alpha_p) + d_j(\alpha_p) \dot{\ell}_j(\alpha_p) } {1+d_j(\alpha_p) \ell_j(\alpha_p) }, \\
\end{aligned}
\end{equation*} 

\noindent and therefore, 

\begin{equation*}
\begin{aligned}
\ddot{f}_j(\alpha_p) & = \frac{ \big(\ddot{d}_j(\alpha_p) \ell_j(\alpha_p) + 2 \dot{d}_j(\alpha_p) \dot{\ell}_j(\alpha_p) +d_j(\alpha_p) \ddot{\ell}_j(\alpha_p)\big) \big( 1+d_j(\alpha_p) \ell_j(\alpha_p) \big) } {\big(1+d_j(\alpha_p) \ell_j(\alpha_p) \big)^2}\\
&  - \frac{\big(\dot{d}_j(\alpha_p) \ell(\alpha_p) + d_j(\alpha_p) \dot{\ell}_j(\alpha_p)\big)^2 } {\big(1+d_j(\alpha_p) \ell_j(\alpha_p) \big)^2}.
\end{aligned}
\end{equation*} 

\noindent  From Corollaries \ref{Dd} and \ref{Dell}, we get 
 
\begin{equation*}
\begin{aligned}
\Big |\dot{f}_j(\alpha_p) \Big| & \leq \frac{C_d^{(1)} \eta_{\max} + d_{\max} C_\ell^{(1)} } {1+d_{\min} \mu_{\min} } = C^{(1)}_{f}, \\
\\
\Big |\ddot{f}_j(\alpha_p) \Big| & \leq  \frac{ \big(C^{(2)}_d \eta_{\max} + 2 C_d^{(1)} C^{(1)} _\ell +d_{\max} C^{(2)}_\ell \big) \big( 1+d_{\max} \eta_{\max} \big) } {\big(1+d_{\min} \mu_{\min} \big)^2} \\
&+ \frac {\big(C^{(1)}_d \eta_{\max} + d_{\max} C^{(1)}_\ell \big)^2 } {\big(1+d_{\min} \mu_{\min} \big)^2} = C^{(2)}_f. 
\end{aligned}
\end{equation*} 

\noindent Therefore $ | \ddot{f} _j(r_j(\alpha_p)) | \leq C^{(2)}_f , $ which implies 

\noindent 

\begin{equation*}
\begin{aligned}
\Big| \frac{ \lambda_1 ^ {(m)} (\alpha_p) - \lambda_1 ^ {(m)} (0)}{\alpha_p}  - F_m \Big| & \leq  \frac{1}{m} \sum_{j=1}^{m}\frac{\alpha_p }{2} \Big| \ddot{f}_j (t_j(\alpha_p))\Big|  
&\leq \frac{C^{(2)}_f}{2} \alpha_p.
\end{aligned}
\end{equation*}


\noindent Now $|\dot{f}_j(\alpha_p)| \leq C^{(1)}_f$, gives $|F_m| \leq \frac{1}{m} \sum_{j=1}^{m} |\dot{f}_j(0)| \leq C^{(1)}_f$, for all $m$. So, the sequence $\{F_m\}$ has convergent subsequences. Let for example $ F_{m_h} \to {F}$, for some sub-sequence $\{m_h\}$. Then 

\begin{equation*}
\begin{aligned}
\Big | \frac{ \lambda^{(m_h)}_1  (\alpha_p) - \lambda^{(m_h)}_1  (0)}{\alpha_p}  - F_{m_h} \Big | &\leq \frac{C^{(2)}_f}{2} \alpha_p, 
\end{aligned}
\end{equation*}

\noindent for all $h\geq 1$. So, letting $h \to \infty $, and letting $ \alpha_p \to 0$ as $p \to \infty$ we get that there exists
 and letting $ \alpha_p \to 0$ as $p \to \infty$ we get that there exists 
\begin{equation*}
\begin{aligned}
\lim_{p \to \infty} \frac{ \lambda_1  (\alpha_p) - \lambda_1  (0)}{\alpha_p}  = F.  
\end{aligned}
\end{equation*} 

\noindent for every sequence $\alpha_1 > \alpha_2 > ... > \alpha_{p} > ... \to 0$ as $p \to \infty $ in $[0, b]$ as in Lemma \ref{alpha_tilde{k}}. Thus, there exists $ F = \lim_{m \to \infty} \frac{1}{m} \sum_{j=1}^{m} \dot{f}_j(0)$. This is true for every subsequence $\{m_h\}$, so for any subsequence we have $ F_{m_h} \to F$. Hence, $ F_{m}$ converges to $F$ as well. This implies that there exists

\begin{equation*}
\begin{aligned}
\lim_{\alpha \to 0} \frac{ \lambda_1  (\alpha) - \lambda_1  (0)}{\alpha}  = F,  
\end{aligned}
\end{equation*}

\noindent so $\lambda_1$ is differentiable at $\alpha =0$ and $\dot{\lambda}_1(0) =F$. This establishes the statement.\qed


\section*{Acknowledgements}{Prof. Luchezar Stoyanov is gratefully acknowledged by the author for his insightful advice, helpful criticism, and assistance. This work was supported by a scholarship from Najran University, Saudi Arabia.}




\end{document}